\documentclass[graybox]{svmult}

\usepackage{mathptmx}       % selects Times Roman as basic font
\usepackage{helvet}         % selects Helvetica as sans-serif font
\usepackage{courier}        % selects Courier as typewriter font
\usepackage{type1cm}        % activate if the above 3 fonts are
\usepackage{makeidx}         % allows index generation
\usepackage{graphicx}        % standard LaTeX graphics tool
\usepackage{multicol}        % used for the two-column index
\usepackage[bottom]{footmisc}% places footnotes at page bottom
\makeindex             % used for the subject index
\begin{document}

\title*{Pythagoras at the Bat}
\author{Steven J Miller, Taylor Corcoran, Jennifer Gossels, Victor Luo and Jaclyn Porfilio}
\institute{Steven J Miller \at Williams College, Williamstown, MA 01267, \email{sjm1@williams.edu, Steven.Miller.MC.96@aya.yale.edu}
\and Taylor Corcoran \at The University of Arizona, Tucson, AZ 85721
\and Jennifer Gossels \at Princeton University, Princeton, NJ 08544
\and Victor Luo \at Williams College, Williamstown, MA 01267
\and Jaclyn Porfilio \at Williams College, Williamstown, MA 01267}
\maketitle

\abstract*{} The Pythagorean formula is one of the most popular ways to measure the true ability of a team. It is very easy to use, estimating a team's winning percentage from the runs they score and allow. This data is readily available on standings pages; no computationally intensive simulations are needed. Normally accurate to within a few games per season, it allows teams to determine how much a run is worth in different situations. This determination helps solve some of the most important economic decisions a team faces: How much is a player worth, which players should be pursued, and how much should they be offered. We discuss the formula and these applications in detail, and provide a theoretical justification, both for the formula as well as simpler linear estimators of a team's winning percentage. The calculations and modeling are discussed in detail, and when possible multiple proofs are given. We analyze the 2012 season in detail, and see that the data for that and other recent years support our modeling conjectures. We conclude with a discussion of work in progress to generalize the formula and increase its predictive power  \emph{without} needing expensive simulations, though at the cost of requiring play-by-play data.

%\abstract{Each chapter should be preceded by an abstract (10--15 lines long) that summarizes the content. The abstract will appear \textit{online} at \url{www.SpringerLink.com} and be available with unrestricted access. This allows unregistered users to read the abstract as a teaser for the complete chapter. As a general rule the abstracts will not appear in the printed version of your book unless it is the style of your particular book or that of the series to which your book belongs.\newline \indent Please use the 'starred' version of the new Springer \texttt{abstract} command for typesetting the text of the online abstracts (cf. source file of this chapter template \texttt{abstract}) and include them with the source files of your manuscript. Use the plain \texttt{abstract} command if the abstract is also to appear in the printed version of the book.}

%%% COMMANDS NEEDED %%%%%
\newcommand{\ncr}[2]{{#1 \choose #2}}
\newcommand{\rl}{{\rm R}_{\rm ave}}
\newcommand{\rt}{{\rm R}_{\rm total}}
\newcommand{\lwp}{\rm WP}
\newcommand{\rs}{{\rm RS}}
\newcommand{\ra}{{\rm RA}}
\newcommand{\rso}{{\rm RS_{\rm obs}}}
\newcommand{\rao}{{\rm RA_{\rm obs}}}

%%%%%%%%%%%%%%%%%%%%%%%%%%%%%%%%%%%%%%%%%%%%%%%%%%%%%%%%%%%%%%%%%%%%%%%%%%%%%%%%%%%%%%%%%%%%%%%%%%%%%%%%%%%%%%%%%%%%%%%%%%%%%%%%%%%%%%%%%%%%%%%
%%%%%%%%%%%%%%%%%%%%%%%%%%%%%%%%%%%%%%%%%%%%%%%%%%%%%%%%%%%%%%%%%%%%%%%%%%%%%%%%%%%%%%%%%%%%%%%%%%%%%%%%%%%%%%%%%%%%%%%%%%%%%%%%%%%%%%%%%%%%%%%
%%%%%%%%%%%%%%%%%%%%%%%%%%%%%%%%%%%%%%%%%%%%%%%%%%%%%%%%%%%%%%%%%%%%%%%%%%%%%%%%%%%%%%%%%%%%%%%%%%%%%%%%%%%%%%%%%%%%%%%%%%%%%%%%%%%%%%%%%%%%%%%
\section{Introduction}\label{sec:pythagintro}

In the classic movie \emph{Other People's Money}, New England Wire and Cable is a firm whose parts are worth more than the whole. Danny Devito's character, Larry the Liquidator, recognizes this and tries to take over the company, with the intent on breaking it up and selling it piecemeal. Gregory Peck plays Jorgy, the owner of the firm, who gives an impassioned defense to the stockholders at a proxy battle about traditional values and the golden days ahead. In the climatic conclusion, Larry the Liquidator responds to Jorgy's speech which painted him a heartless predator who builds nothing and cares for no one but himself. Larry says

\begin{quote} Who cares? I'll tell you. Me. I'm not your best friend. I'm your only friend. I don't make anything? I'm making you money. And lest we forget, that's the only reason any of you became stockholders in the first place. You want to make money! You don't care if they manufacture wire and cable, fried chicken, or grow tangerines! You want to make money! I'm the only friend you've got. I'm making you money. \end{quote}

While his speech is significantly longer than this snippet, the scene in general and the lines above in particular highlight one of the most important problems in baseball, one which is easily forgotten. In the twenty-first century massive computation is possible. Data is available in greater quantities than ever before; it can be analyzed, manipulated, and analyzed again thousands of times a second. We can search for small connections between unlikely events. This is especially true in baseball, as there has been an explosion of statistics that are studied and quoted, both among the experts and practitioners as well as the everyday fan. The traditional metrics are falling out of favor, being replaced by a veritable alphabet soup of acronyms. There are so many statistics now, and so many possibilities to analyze, that good metrics are drowned out in poor ones. We need to determine which ones matter most.

In this chapter we assume a team's goal is to win as many games as possible given a specified amount of money to spend on players and related items. This is a reasonable assumption from the point of view of general managers, though it may not be the owner's goal (which could range from winning at all costs to creating the most profitable team). In this case, Devito's character has very valuable advice: The goal is to win games. We don't care if it's by winning shoot-outs 12-10 in thirteen innings, or by eking out a win in a 1-0 pitcher's duel. We want to win games.

In this light, we see that sabermetrics is a dear friend. While there are many items we could study, we focus on the value of a run (both a run created and a run saved). We have a two-stage process. We need to determine how much each event is worth in terms of creating a run, and then we need to extract how much a run is worth. Obviously these are not constant values; a run is worth far more in a 2-1 game than in a 10-1 match. We focus below entirely on the value of a run. We thus completely ignore the first item above, namely how much each event contributes to scoring.

Our metric for determining the worth of a run is Bill James' Pythagorean Won-Loss formula: If a team scores $\rs$ runs while allowing $\ra$, then their winning percentage is approximately $\frac{\rs^\gamma}{\rs^\gamma + \ra^\gamma}$. Here $\gamma$ is an exponent whose value can vary from sport to sport (as well as from era to era within that sport). James initially took $\gamma$ to be 2, which is the source of the name as the formula is reminiscent of the sum of squares from the Pythagorean theorem. \emph{Note that instead of using the total runs scored and allowed we could use the average number per game, as such a change rescales the numerator and the denominator by the same amount.}

In this chapter we discuss previous work providing a theoretical justification for this formula, talk about future generalizations, and describe its implications in one of the most important economics problems confronted by a baseball team: How much is a given player worth? While much of this chapter has appeared in journals, we hope that by combining everything in one place and doing the calculations in full detail and in as elementary a way as possible that we will increase the visibility of this method, and provide support for the role of mathematical modeling in sabermetrics.

Before delving into the derivation, it's worth remarking on why such a derivation is important, and what it can teach us. In \emph{An Enquiry Concerning Human Understanding} (1772), David Hume wrote:

\begin{quote} The contrary of every matter of fact is still possible, because it can never imply a contradiction, and is conceived by the mind with the same facility and distinctness, as if ever so conformable to reality. That the sun will not rise tomorrow is no less intelligible a proposition, and implies no more contradiction, than the affirmation, that it will rise. We should in vain, therefore, attempt to demonstrate its falsehood. Were it demonstratively false, it would imply a contradiction, and could never be distinctly conceived by the mind. \end{quote}

Hume's warning complements our earlier quote, and can be summarized by saying that just because the sun rose yesterday we cannot conclude that it will rise today. Sabermetricians frequently find quantities that appear to be well correlated with desirable outcomes; however, there is a real danger that the correlation will not persist in the future as past performance is no guarantee of future performance. (This lesson has been painfully learned by many chartists on Wall Street.) Thus we must be careful in making decisions based on regressions and other calculations. If we find a relationship, we want some \emph{reason} to believe it will continue to hold.

We are therefore led to creating mathematical models with reasonable assumptions; thus the point of this chapter is to develop predictive mathematical models to complement inferential techniques. The advantage of this approach is that we now have a reason to believe the observed pattern will continue, as we can now point to an explanation, a reason. We will find such a model for baseball, which has the Pythagorean formula, initially a numerical observation by James that seemed to do a good job year after year, as a consequence.

The Pythagorean formula has a rich history; almost any sabermetrics book references it at some point. It is necessary to limit our discussion to just some of its aspects. As the economic consequences to a team from better predictive power are clear, we concentrate on the mathematical issues. Thus explaining how mathematical models can lead to closed-form expressions, which can solve real world problems, is our main goal. We begin in \S\ref{sec:pythagtheory} with some general comments on the statistic. We describe a reasonable mathematical model in the next section, show the Pythagorean formula is a consequence, and then give a mathematical proof in \S\ref{sec:pythagproof}. In \S\ref{sec:pythagapplication} we examine some consequences, in particular how much a run created or saved is worth at different production levels, and in \S\ref{sec:pythagverification} we analyze data from several seasons to see how well our model and the formula do. Next we examine linear predictors for a team's winning percentage, and show how they follow from linearizing the Pythagorean formula. We end by discussing current, ongoing research into generalizing the Pythagorean formula.
%, and give detailed derivations of some needed results in the appendices.

%%%%%%%%%%%%%%%%%%%%%%%%%%%%%%%%%%%%%%%%%%%%%%%%%%%%%%%%%%%%%%%%%%%%%%%%%%%%%%%%%%%%%%%%%%%%%%%%%%%%%%%%%%%%%%%%%%%%%%%%%%%%%%%%%%%%%%%%%%%%%%%
%%%%%%%%%%%%%%%%%%%%%%%%%%%%%%%%%%%%%%%%%%%%%%%%%%%%%%%%%%%%%%%%%%%%%%%%%%%%%%%%%%%%%%%%%%%%%%%%%%%%%%%%%%%%%%%%%%%%%%%%%%%%%%%%%%%%%%%%%%%%%%%
%%%%%%%%%%%%%%%%%%%%%%%%%%%%%%%%%%%%%%%%%%%%%%%%%%%%%%%%%%%%%%%%%%%%%%%%%%%%%%%%%%%%%%%%%%%%%%%%%%%%%%%%%%%%%%%%%%%%%%%%%%%%%%%%%%%%%%%%%%%%%%%

\section{General Comments}\label{sec:pythagtheory}

Before discussing why the Pythagorean formula should be true, it's worth commenting on the form it has, both in its present state and its debut back in Bill James' \emph{1981 Baseball Abstract} [6]. Remember it says that a team's winning percentage should be $\frac{\rs^\gamma}{\rs^\gamma + \ra^\gamma}$, with $\gamma$ initially taken as 2 but now typically taken to be around 1.83. One is struck by how easy the formula is to state and to use, especially in the original incarnation. All we need is to know the average number of runs scored and allowed, and the ratio can be found on a simple calculator.

Of course, back in the '80s this wouldn't be entirely true for someone watching at home if $\gamma$ were not 2, though the additional algebra is slight and not even noticeable on modern calculators, computers and even phones. One of the great values of this statistic is just how easy it is to calculate, which is one of the reasons for its popularity. You can easily approximate how much better you would do if you scored 10 more runs, or allowed 10 fewer, which we do later in Figure \ref{fig:pythaggainsave}. We can do this as we have a \emph{simple, closed form expression} for our winning percentage in terms of just three parameters: average runs scored, average runs allowed, and an exponent $\gamma$.

This is very different than the multitude of Monte Carlo simulations which try to predict a team's record. These require detailed statistics on batters and pitchers and their interactions. Depending on how good and involved the algorithm is, we may need everything from how many pitches a batter sees per appearance to the likelihood of a runner advancing from first to third on a single hit to right field. While this data is available, it takes time to simulate thousands of games. Further, every small change in a team requires an entirely new batch of simulations. With the Pythagorean formula, we can immediately determine the impact of a player \emph{if} we have a good measure of how many runs they will contribute or save.

Of course, as with most things in life there are trade-offs. While the closed-form nature of the Pythagorean formula allows us to readily measure the impact of players, it indicates a major defect that should be addressed. Baseball is a complicated game; it is unlikely that all the subtleties and issues can be distilled into one simple formula involving just three inputs. Admittedly, it is a major challenge to derive a good formula to predict how many runs a player will give a team, and we are ignoring this issue in this chapter; however, it is improbable that any formula as simple as this can capture everything that matters. There are thus several extensions of the Pythagorean formula; we discuss some of these in Sections \S\ref{sec:pythaglinearization} and \S\ref{sec:pythaggeneralization}, as well as outline a program currently being pursued to improve its predictive power.

\section{Pythagorean Formula: Model}\label{sec:pythagmodel}

There are many ways to model a baseball game. The more sophisticated the model, the more features can be captured, though added power comes at a cost. The cost varies from increased run-time to requiring massively more data. We give a very simple model for a baseball game, and show the Pythagorean formula is a consequence. Of course, the simplicity of our model strongly suggests that it cannot be the full story. We return to that issue in \S\ref{sec:pythaggeneralization}, and content ourselves here with the simple case. The hope is that this simplified model of baseball is nevertheless powerful enough to capture the main features and yield a reasonably good predictive statistic. See the paper of Hammond, Johnson and Miller [4] for other approaches to modeling baseball games and winning percentages. Specifically, they look at James' log5 method, which also appeared in his 1981 abstract [6]. There he estimates the probability a team with winning percentage $a$ beats a team with winning percentage $b$ by $\frac{a(1-b)}{a(1-b)+(1-a)b}$. Interestingly, the Pythagorean formula with exponent 2 follows by taking $a = {\rm RS}/({\rm RS} + {\rm RA})$ and $a = {\rm RA}/({\rm RS} + {\rm RA})$, with ${\rm RS}$ the average number of runs scored and ${\rm RA}$ the average number of runs allowed.

The following model and derivation first appeared in work by the first author in [9], who introduced using a Weibull distribution to model run production. The Weibull distribution is extensively used in statistics, arising in many problems in survival analysis (see [12] for a good description of the Weibull's properties and applications). The reason a Weibull distribution is able to model well so many different data sets is that it is a three parameter distribution, with probability density function  \begin{equation} f(x;\alpha,\beta,\gamma) \ = \  \frac{\gamma}{\alpha}\ ((x-\beta)/\alpha)^{\gamma-1}\ e^{- ((x-\beta)/\alpha)^{\gamma}} \end{equation} if $x \ge \beta$ and 0 otherwise. Here $\alpha$, $\beta$ and $\gamma$ are the three parameters of the distribution. The effect of $\beta$ is to shift the entire distribution along the real line; essentially it determines the starting point. In our investigations $\beta$ will always be $-1/2$, for reasons that will become clear. Next is $\alpha$, which adjusts the scale of the distribution but not the shape; as $\alpha$ increases the distribution becomes more spread out.

The reason that $\alpha$ and $\beta$ do not alter the shape of the distribution is that, for any distribution with finite mean and variance, we can always rescale it to have mean zero and variance 1 (or, more generally, any mean and any positive variance). Thus all $\alpha$ and $\beta$ do are adjust these two quantities. It is $\gamma$ that is the most important, as different values of $\gamma$ lead to very different shapes. We illustrate this in Figure \ref{fig:pythagweibullgamma1to4}. For definiteness, we may rescale and assume $\alpha = 1$ and $\beta = 0$; we see how the distribution changes as $\gamma$ ranges from 1 to 2.

\begin{figure}[h]
\begin{center}
\scalebox{.9125}{\includegraphics{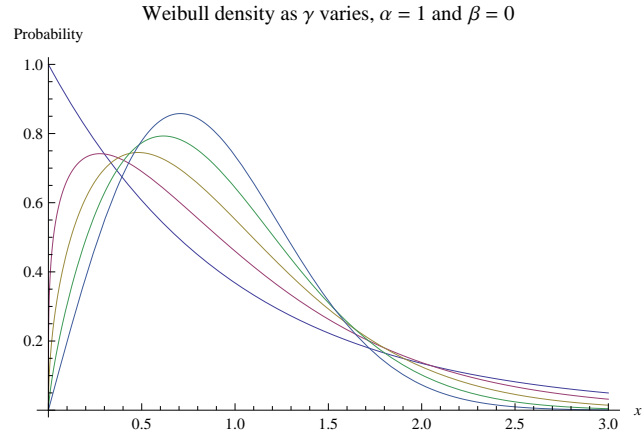}}
\caption{\label{fig:pythagweibullgamma1to4} The changing probabilities of a family of Weibulls with $\alpha = 1$, $\beta = 0$, and $\gamma \in \{1, 1.25, 1.5, 1.75, 2\}$; $\gamma = 1$ corresponds to the exponential distribution, and increasing $\gamma$ results in the bump moving rightward.}
\end{center}
\end{figure}

%weibull[x_, a_, b_, c_] := (c/a) ((x - b)/a)^(c - 1.) Exp[-((x - b)/a)^c];
%NIntegrate[weibull[x, .3, .54, 1.4], {x, .54, Infinity}]

%Plot[{weibull[x, 1, 0, 1], weibull[x, 1, 0, 1.25],
%  weibull[x, 1, 0, 1.5], weibull[x, 1, 0, 1.75],
%  weibull[x, 1, 0, 2]}, {x, 0, 3}, AxesLabel -> {x, "Probability"},
% PlotLabel ->
%  "Weibull density as \[Gamma] varies, \[Alpha] = 1 and \[Beta] = 0"]

We are now ready to state our model. After listing our assumptions we discuss why these choices were made, and their reasonableness. \emph{Remember, as remarked earlier, that in the Pythagorean formula it makes no difference if we use the total runs or the average per game, as rescaling changes the numerator and the denominator by the same multiplicative factor, and hence has no effect.}

\begin{svgraybox}\textbf{Assumptions for modeling a baseball game:} The average number of runs a team scores per game, denoted $\rs$, and the average number of runs allowed per game, denoted $\ra$, are random variables drawn independently from Weibull distributions with $\beta = -1/2$ and the same $\gamma$.
\end{svgraybox}

These assumptions clearly require discussion, as they cannot be right. The first issue is that we are modeling runs scored and allowed by continuous random variables and not discrete random variables. While earlier work in the field used discrete random variables (especially geometric or Poisson), the difficulty with these approaches is that it is hard to obtain tractable, closed form expressions for the probability a team scores more runs than it allows and hence wins a game. The reason is that calculus is unavailable in this case. Another way to put it is that while many people have continued in mathematics to Calculus III or IV, no one goes similarly far in classes on summation. In general, we do not have good formulas for sums, but through calculus we do have nice expressions for integrals. While the model allows for the Red Sox to beat the Yankees $\pi$ to $e$, we must accept this if we want to be able to use calculus.

The next assumption is that these random variables are drawn from Weibull distributions. There are two reasons for this. One is that the Weibull distributions, due to their shape parameter $\gamma$, are an extremely flexible family and are capable of fitting many one-hump distributions (i.e., distributions that go up and then go down). The second, and far more important, is that calculations with the Weibull are exceptionally tractable and lead to closed form expressions. This should be compared to similar and earlier work of Hein Hundel [5], which the author learned of from the Wikipedia entry \emph{Pythagorean expectation} [11]. In particular, the mean $\mu_{\alpha,\beta,\gamma}$ and the variance $\sigma^2_{\alpha,\beta,\gamma}$ of the Weibull are readily computed: \begin{eqnarray} \mu_{\alpha,\beta,\gamma} &\ =\ & \alpha \Gamma\left(1+\gamma^{-1}\right) + \beta \nonumber\\ \sigma^2_{\alpha,\beta,\gamma}&\ =\ & \alpha^2 \Gamma\left(1+2\gamma^{-1}\right)  - \alpha^2 \Gamma\left(1+\gamma^{-1}\right)^2.\end{eqnarray} Here $\Gamma(s)$ is the Gamma function, defined for the real part of $s$ positive by \begin{equation} \Gamma(s)\ =\ \int_0^\infty e^{-u} u^{s-1} {\rm d} u. \end{equation} The Gamma function is the continuous generalization of the factorial function, as for $n$ a non-negative integer we have $\Gamma(n+1) = n!$.

The reason Weibulls lead to such tractable calculations is that if $X$ is a random variable drawn from a Weibull with parameters $\alpha, \beta$ and $\gamma$, then $X^{1/\gamma}$ is exponentially distributed with parameter $\alpha^\gamma$. Therefore a simple change of variables leads to simple integrals of exponentials, which can be done in closed form. Due to the importance of this calculation, we give full details for the computation of the mean in Appendix \ref{sec:pythagmeanweibull} (a similar calculation determines the variance). The point is that when there are several alternatives to use, certain choices are more tractable and should be incorporated. We discuss how to handle more general distributions while preserving the all-important closed form nature of the solution in \S\ref{sec:pythaggeneralization}.

The next issue is our assumption that $\beta = -1/2$. This choice is to facilitate comparisons to the discrete scoring in baseball. Using the above calculations for the mean, if $\beta$ and $\gamma$ are fixed we can determine $\alpha$ so that the mean of our Weibull matches the observed average runs scored (or allowed) per game. We can use the Method of Least Squares or the Method of Maximum Likelihood to find the best fit parameters $\alpha, \beta, \gamma$ to the observed data. In doing so, we need to deal with the fact that our data is discrete. By taking $\beta = -1/2$, we are breaking the data into bins $[-\frac12, \frac12)$, $[\frac12, 1\frac12)$, $[1\frac12, 2\frac12)$ and so on. Notice that the \emph{centers} of these bins are, respectively, 0, 1, 2, $\dots$. This is no accident, and in fact is the reason we chose $\beta$ as we did. By taking $\beta =-1/2$ the possible integer scores are in the \emph{middle} of each bin. If we took $\beta = 0$, as might seem more natural, then these values would lie at the endpoints of the bins, which would cause issues in determining the best fit values.

The final issue is that we are assuming runs scored and runs allowed are independent. This of course cannot be true, for the very simple reason that baseball games cannot end in a tie! Thus if we know the Orioles scored 5 runs against the Red Sox, then we know the Sox ended the game with some number other than 5. There are a plethora of other obvious issues with this assumption, ranging from if you have a large lead late in the game you might rest your better players and take a chance on a weaker pitcher, to bringing in your closer to protect the lead in a tight game. That said, an analysis of the data shows that on average these issues cancel each other out, and that subject to being different the runs scored and allowed behave as if they are statistically independent. The interesting feature here is that we cannot use a standard $r \times c$ contingency table analysis as these two values cannot be equal. This leads to an iterative procedure taking into account these \emph{structural zeros} (values of the table that are inaccessible), which is described in Appendix \ref{sec:pythagappendixindstructuralzeros}.

We end this section by describing the calculation that yields the Pythagorean formula, and remarking on why we have chosen to model the runs with Weibull distributions. Let $X$ be a random variable drawn from a Weibull with parameters $\alpha_{\rs}, \beta = -1/2$ and $\gamma$, representing the number of runs a team scores on average. Similarly, let $Y$ be a random variable drawn from a Weibull with parameters $\alpha_{\ra}, \beta = -1/2$ and $\gamma$, representing the number of runs a team allows on average. Notice we have the same $\gamma$ for $X$ and $Y$, and we choose $\alpha_{\rs}$ and $\alpha_{\ra}$ so that the mean of $X$ is the observed average number of runs scored per game, $\rs$, and the mean of $Y$ is the observed average number of runs allowed per game, $\ra$. Thus \begin{eqnarray}\label{eq:genweibmeans} \alpha_\rs & \ = \ & \frac{\rs-\beta}{\Gamma(1+\gamma^{-1})}, \ \ \ \ \ \ \alpha_\ra  \ = \  \frac{\ra-\beta}{\Gamma(1+\gamma^{-1})}.\end{eqnarray} To determine our team's winning percentage we just need to calculate the probability that $X$ exceeds $Y$: \begin{equation} \mbox{Prob}(X > Y) \ = \ \int_{x=\beta}^\infty \int_{y=\beta}^x f(x;\alpha_\rs,\beta,\gamma) f(y;\alpha_\ra,\beta,\gamma) {\rm d} y\; {\rm d} x. \end{equation}

For general probability densities $f$ the above double integral is intractable (as can be seen in Hundel's work, where he used the log-normal distribution). As we'll see in the next section, the Weibull distribution leads to very simple integrals which can be evaluated in closed form. This is not am accidental, fortuitous coincidence. When first investigating this problem, Miller began by choosing $f$'s that led to nice double integrals which could be computed in closed form; thus the choice of the Weibull came not from looking at the data but from looking at the integration! The first $f$ Miller chose was an exponential distribution, which turns out to be a Weibull with $\gamma = 1$. Next, Miller chose a Rayleigh distribution, which is a Weibull with $\gamma = 2$. (As a number theorist working in random matrix theory, which is often used to model the energy levels of heavy nuclei, the Rayleigh distribution was one Miller encountered frequently in his research and reading, as it approximates the spacings between energy levels of heavy nuclei.) It was only after computing the answer in both these cases that Miller realized the two densities fit into a nice family, and did the calculation for general $\gamma$.

%%%%%%%%%%%%%%%%%%%%%%%%%%%%%%%%%%%%%%%%%%%%%%%%%%%%%%%%%%%%%%%%%%%%%%%%%%%%%%%%%%%%%%%%%%%%%%%%%%%%%%%%%%%%%%%%%%%%%%%%%%%%%%%%%%%%%%%%%%%%%%%
%%%%%%%%%%%%%%%%%%%%%%%%%%%%%%%%%%%%%%%%%%%%%%%%%%%%%%%%%%%%%%%%%%%%%%%%%%%%%%%%%%%%%%%%%%%%%%%%%%%%%%%%%%%%%%%%%%%%%%%%%%%%%%%%%%%%%%%%%%%%%%%
%%%%%%%%%%%%%%%%%%%%%%%%%%%%%%%%%%%%%%%%%%%%%%%%%%%%%%%%%%%%%%%%%%%%%%%%%%%%%%%%%%%%%%%%%%%%%%%%%%%%%%%%%%%%%%%%%%%%%%%%%%%%%%%%%%%%%%%%%%%%%%%

\section{Pythagorean Formula: Proof}\label{sec:pythagproof}

We now finally prove the Pythagorean formula, which we first state explicitly as a theorem. For completeness, we restate our assumptions.

\begin{theorem}[Pythagorean Won-Loss Formula]\label{thm:pythmweibull}
Let the runs scored and runs allowed per game be two independent random variables drawn from Weibull
distributions with parameters $(\alpha_\rs,\beta,\gamma)$ and $(\alpha_\ra,\beta,\gamma)$ respectively, where $\alpha_\rs$ and $\alpha_\ra$ are
chosen so that the means are $\rs$ and $\ra$; in applications $\beta = -1/2$. Then
\begin{equation}\label{eq:pythagammap} \mbox{\rm Won-Loss
Percentage}(\rs,\ra,\beta,\gamma) \ = \ \frac{(\rs-\beta)^\gamma}{(\rs-\beta)^\gamma
+ (\ra-\beta)^\gamma}. \end{equation}
\end{theorem}

\begin{proof} Let $X$ and $Y$ be independent random variables with Weibull
distributions $(\alpha_\rs,\beta,\gamma)$ and $(\alpha_\ra,\beta,\gamma)$
respectively, where $X$ is the number of runs scored and $Y$ the
number of runs allowed per game. Recall from (\ref{eq:genweibmeans}) that \begin{eqnarray} \alpha_\rs & \ = \ &
\frac{\rs-\beta}{\Gamma(1+\gamma^{-1})}, \ \ \ \ \ \ \ \alpha_\ra \ = \ \frac{\ra-\beta}{\Gamma(1+\gamma^{-1})}.\end{eqnarray}

We need only calculate the probability that $X$ exceeds $Y$. Below we constantly use the integral of a probability density is $1$ (for example, in moving from the second to last to the final line). We
have \begin{eqnarray} & & \mbox{Prob}(X > Y) \ = \ \int_{x=\beta}^\infty \int_{y=\beta}^x f(x;\alpha_\rs,\beta,\gamma) f(y;\alpha_\ra,\beta,\gamma) {\rm d} y\; {\rm d} x \nonumber\\
& & = \ \int_{x=\beta}^\infty\int_{y=\beta}^x \frac{\gamma}{\alpha_\rs} \left(\frac{x-\beta}{\alpha_{RS}}\right)^{\gamma-1} e^{-((x-\beta)/\alpha_\rs)^\gamma} \frac{\gamma}{\alpha_\ra}\left(\frac{y-\beta}{\alpha_{\ra}}\right)^{\gamma-1} e^{- ((y-\beta)/\alpha_\ra)^\gamma} {\rm d} y\; {\rm d} x \nonumber\\ & & = \
\int_{x=0}^\infty\frac{\gamma}{\alpha_\rs} \left(\frac{x}{\alpha_{RS}}\right)^{\gamma-1} e^{-(x/\alpha_\rs)^\gamma} \left[
\int_{y=0}^{x} \frac{\gamma}{\alpha_\ra}\left(\frac{y}{\alpha_{\ra}}\right)^{\gamma-1} e^{-
(y/\alpha_\ra)^\gamma} {\rm d} y \right] {\rm d} x  \nonumber\\ & & = \ \int_{x=0}^\infty\frac{\gamma}{\alpha_\rs}
\left(\frac{x}{\alpha_{RS}}\right)^{\gamma-1} e^{-(x/\alpha_\rs)^\gamma} \left[1
- e^{-(x/\alpha_{\ra})^\gamma}\right] {\rm d} x  \nonumber\\ & & = \ 1 -
\int_{x=0}^\infty\frac{\gamma}{\alpha_\rs}
\left(\frac{x}{\alpha_{RS}}\right)^{\gamma-1} e^{-(x/\alpha)^\gamma} {\rm d} x,\end{eqnarray}
where we have set \begin{equation} \frac1{\alpha^\gamma} \ = \ \frac1{\alpha_\rs^\gamma} +
\frac{1}{\alpha_\ra^\gamma} \ = \ \frac{\alpha_\rs^\gamma +
\alpha_\ra^\gamma}{\alpha_\rs^\gamma \alpha_\ra^\gamma}.\end{equation} The above tells us that we are essentially integrating a new Weibull whose parameter $\alpha$ is given by the above relation; expressions like this are common (see for example center of mass calculations, or adding resistors in parallel).  Therefore
\begin{eqnarray}\label{eq:derivweibpythag1} \mbox{Prob}(X
> Y) & \ = \ & 1 - \frac{\alpha^\gamma}{\alpha_\rs^\gamma} \int_{0}^\infty
\frac{\gamma}{\alpha}
\left(\frac{x}{\alpha}\right)^{\gamma-1} e^{(x/\alpha)^\gamma} {\rm d} x \nonumber\\
& = & 1 - \frac{\alpha^\gamma}{\alpha_\rs^\gamma} \nonumber\\ & = & 1 -
\frac1{\alpha_\rs^\gamma} \frac{\alpha_\rs^\gamma \alpha_\ra^\gamma}{\alpha_\rs^\gamma +
\alpha_\ra^\gamma} \nonumber\\ & = & \frac{\alpha_\rs^\gamma}{\alpha_\rs^\gamma +
\alpha_\ra^\gamma}.\end{eqnarray} Substituting the relations for $\alpha_\rs$ and
$\alpha_\ra$ of (\ref{eq:genweibmeans}) into
(\ref{eq:derivweibpythag1}) yields \begin{eqnarray} \mbox{Prob}(X
> Y) & \ = \ &  \frac{(\rs-\beta)^\gamma}{(\rs-\beta)^\gamma + (\ra-\beta)^\gamma}, \end{eqnarray}
which completes the proof of Theorem \ref{thm:pythmweibull}, the Pythagorean formula. \hfill Q.E.D.
\end{proof}

%%%%%%%%%%%%%%%%%%%%%%%%%%%%%%%%%%%%%%%%%%%%%%%%%%%%%%%%%%%%%%%%%%%%%%%%%%%%%%%%%%%%%%%%%%%%%%%%%%%%%%%%%%%%%%%%%%%%%%%%%%%%%%%%%%%%%%%%%%%%%%%
%%%%%%%%%%%%%%%%%%%%%%%%%%%%%%%%%%%%%%%%%%%%%%%%%%%%%%%%%%%%%%%%%%%%%%%%%%%%%%%%%%%%%%%%%%%%%%%%%%%%%%%%%%%%%%%%%%%%%%%%%%%%%%%%%%%%%%%%%%%%%%%
%%%%%%%%%%%%%%%%%%%%%%%%%%%%%%%%%%%%%%%%%%%%%%%%%%%%%%%%%%%%%%%%%%%%%%%%%%%%%%%%%%%%%%%%%%%%%%%%%%%%%%%%%%%%%%%%%%%%%%%%%%%%%%%%%%%%%%%%%%%%%%%

\section{The Pythagorean Formula: Applications}\label{sec:pythagapplication}

It is now time to apply our mathematical models and results to the central economics issue of this chapter: In each situation, how much is a run worth? We content ourselves with answering this from the point of view of the season. Thus if we score $x$ runs and allow $y$, and we have a player who increases our run production by $s$, how much is that worth? Similarly, how much would they be worth if they prevented $s$ runs from scoring?

We answer this question not in dollars, but in additional games won or lost. Translating the number of wins per season into dollar amounts is a fascinating and obviously important question, which the interested reader is encouraged to pursue. A good resource is Nate Silver's chapter ``Is Alex Rodriguez Overpaid'' in \emph{Baseball Between the Numbers: Why Everything You Know About the Game Is Wrong} [10]. There are also numerous insightful blog posts, such as Phil Birnbaum's ``Sabermetric Research: Saturday, April 24, 2010'' (see [1]). In this chapter we concern ourselves with determining the number of wins gained or lost, which these and other sources can convert to monetary amounts. As not all wins are worth the same (going from 65 to 75 wins doesn't alter the fact that the season was a bust, but going from 85 wins to 95 wins almost surely punches your ticket to the playoffs), it is essential that we can determine changes from any state.

In Figure \ref{fig:pythaggainsave} we plot the addition wins per season with $\gamma = 1.83$ and $s=10$. We plot around a league average of 700 runs scored per season, which was essentially the average in 2012 (see \S\ref{sec:pythaglinearization}). We let $s=10$ as the common adage is every 10 additional runs translates to one more win per season.

\begin{figure}[h]
\begin{center}
\scalebox{.6125}{\includegraphics{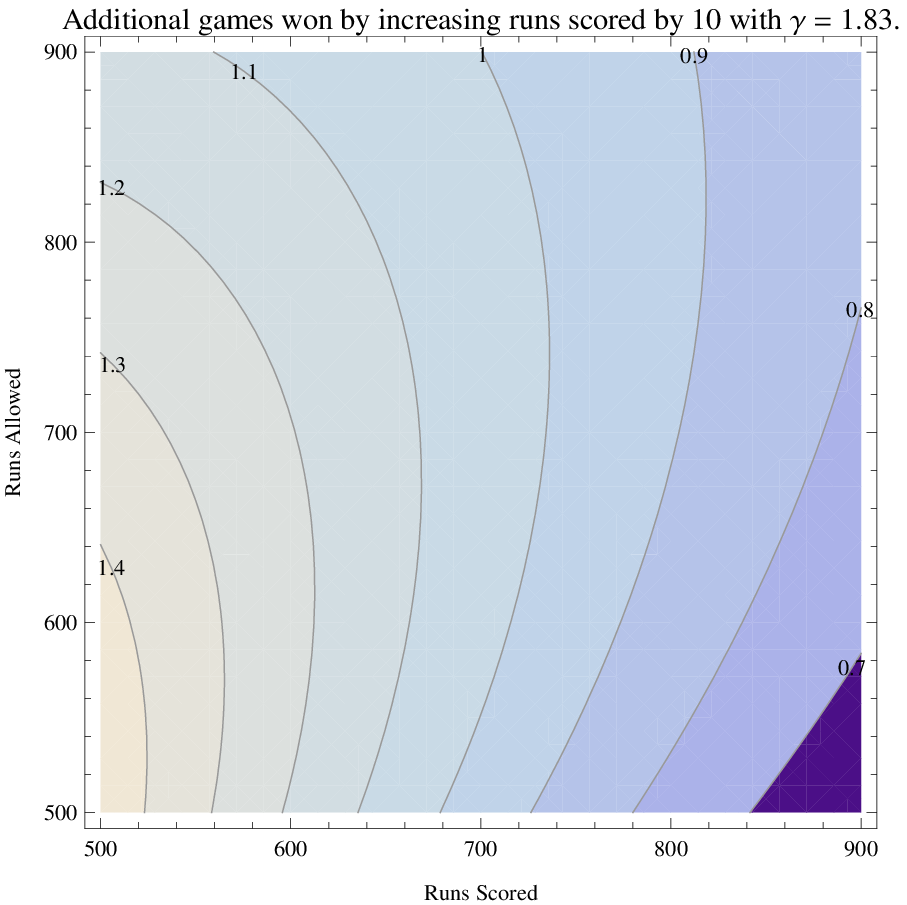}}\ \scalebox{.6125}{\includegraphics{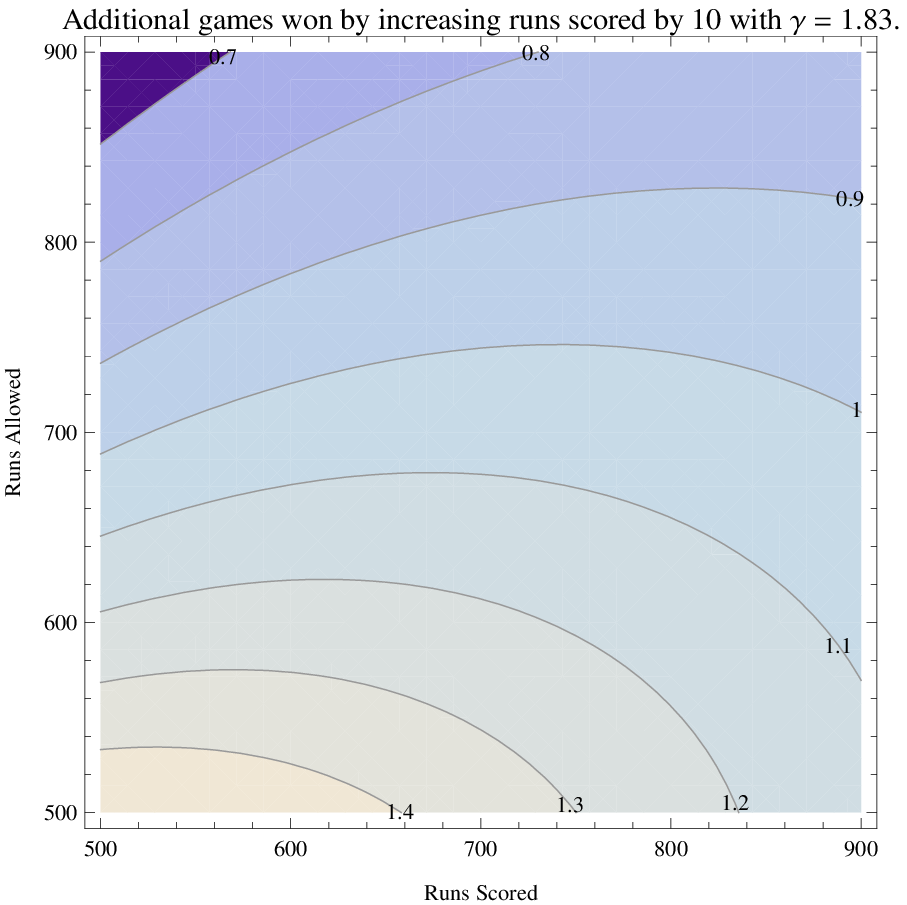}}
\caption{\label{fig:pythaggainsave} The predicted number of additional wins with $\gamma = 1.83$: (left) scoring 10 more per season; (right) preventing 10 more per season. Letting $\mathcal{P}(x,y;\gamma) = x^\gamma / (x^\gamma + y^\gamma)$, the left plot is $\mathcal{P}(x+10,y;1.83) - \mathcal{P}(x,y;1.83)$, while the right is $\mathcal{P}(x,y-10;1.83) - \mathcal{P}(x,y;1.83)$.}
\end{center}
\end{figure}

Not surprisingly, the more runs we score the more valuable preventing runs is to scoring runs, and vice-versa; what is nice about the Pythagorean formula is that it quantifies exactly what this trade-off is. To make it easier to see, in Figure \ref{fig:pythaggainminussave} we plot the difference in wins gained from scoring 10 more runs to wins gained from preventing 10 more runs. The plot is positive in the upper left region, indicating that if our runs scored and allowed places us here then it is more valuable to score runs; in the lower right region the conclusion is the opposite.

\begin{figure}[h]
\begin{center}
\scalebox{.6125}{\includegraphics{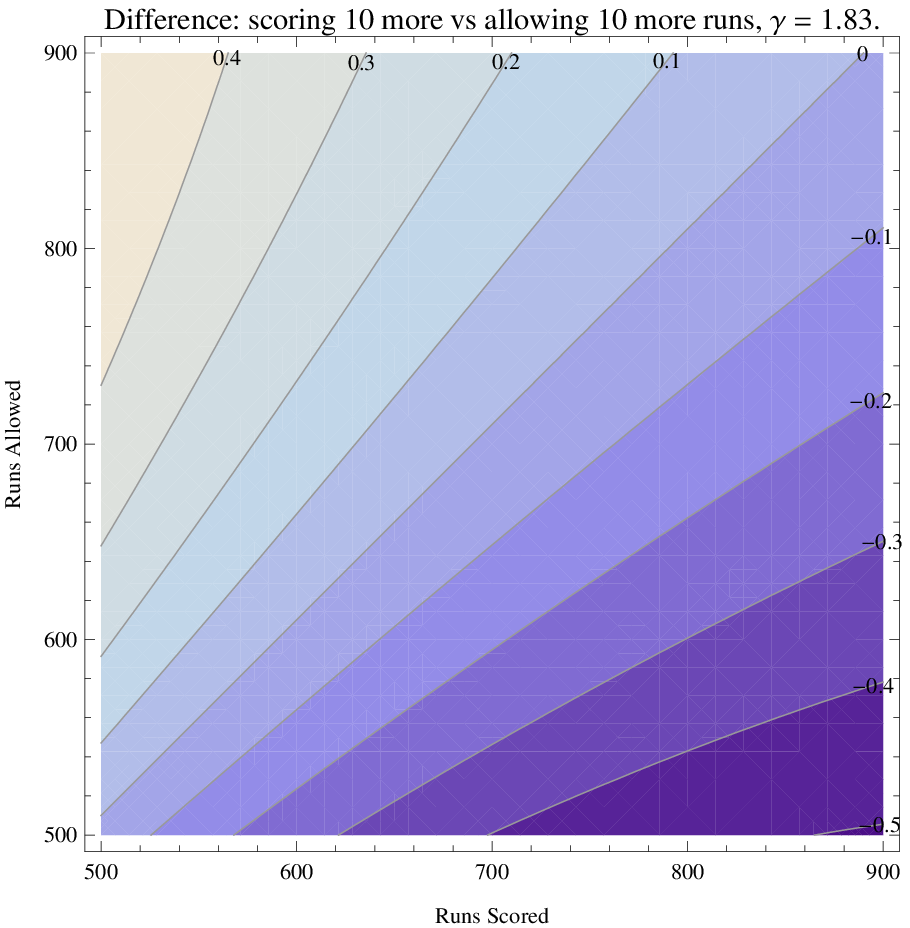}}\ \scalebox{.6125}{\includegraphics{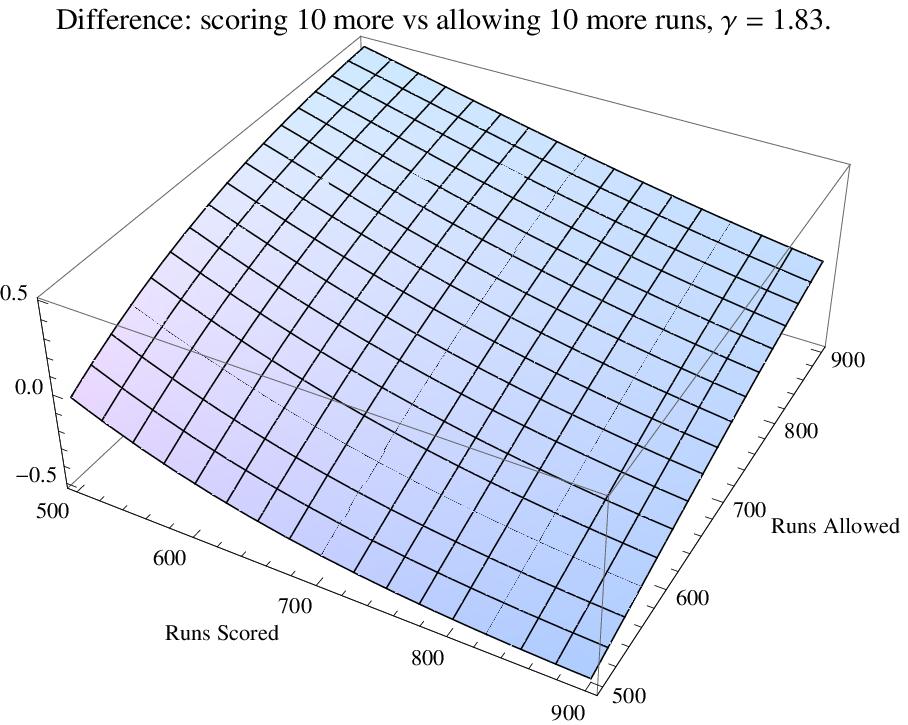}}
\caption{\label{fig:pythaggainminussave} The difference in the predicted number of additional wins with $\gamma = 1.83$ from scoring 10 more per season versus preventing 10 more per season. Letting $\mathcal{P}(x,y;\gamma) = x^\gamma / (x^\gamma + y^\gamma)$, the difference is $\mathcal{P}(x+10,y;\gamma) -\mathcal{P}(x,y-10;\gamma)$.}
\end{center}
\end{figure}

%%%%%%%%%%%%%%%%%%%%%%%%%%%%%%%%%%%%%%%%%%%%%%%%%%%%%%%%%%%%%%%%%%%%%%%%%%%%%%%%%%%%%%%%%%%%%%%%%%%%%%%%%%%%%%%%%%%%%%%%%%%%%%%%%%%%%%%%%%%%%%%
%%%%%%%%%%%%%%%%%%%%%%%%%%%%%%%%%%%%%%%%%%%%%%%%%%%%%%%%%%%%%%%%%%%%%%%%%%%%%%%%%%%%%%%%%%%%%%%%%%%%%%%%%%%%%%%%%%%%%%%%%%%%%%%%%%%%%%%%%%%%%%%
%%%%%%%%%%%%%%%%%%%%%%%%%%%%%%%%%%%%%%%%%%%%%%%%%%%%%%%%%%%%%%%%%%%%%%%%%%%%%%%%%%%%%%%%%%%%%%%%%%%%%%%%%%%%%%%%%%%%%%%%%%%%%%%%%%%%%%%%%%%%%%%

\section{The Pythagorean Formula: Verification}\label{sec:pythagverification}

We have two goals in this section. First, we want to show our assumption of the runs scored and allowed being drawn from independent Weibulls is reasonable. Second, we want to find the optimal value of $\gamma$, and check the conventional wisdom that the Pythagorean formula is typically accurate to about four games a season.

There are many methods available for such analyses. Two popular ones are the Method of Least Squares, and the Method of Maximum Likelihood. As the two give similar results, we use the Method of Least Squares to attack the independence and distributional questions, and the Method of Maximum Likelihood to estimate $\gamma$ and the error in the formula.

\subsection{Analysis of Independence and Distributional Assumptions}

We use the Method of Least Squares to analyze the 30 teams, which are ordered by the number of overall season wins and by league, from the 2012 season to see how closely our model fits the observed scoring patterns. We briefly summarize the procedure. For each team we find $\alpha_{\rs}, \alpha_{\ra}, \beta$ and $\gamma$ that minimize the sum of squared errors from the runs scored data plus the sum of squared errors from the runs allowed data; instead of the Method of Least Squares we could also use the Method of Maximum Likelihood (discussed in the next subsection), which would return similar values. We always take $\beta = -1/2$ and let $\gamma$ vary among teams (though we could also perform the analysis with the same $\gamma$ for all). We partition the runs data into the bins \begin{equation} [-.5, .5), \ \ [.5, 1.5], \ \ [1.5, 2.5], \ \ \dots, \ \ [8.5, 9.5), \ \ [9.5, 11.5), \ \ [11.5, \infty).\end{equation} Let Bin($k$) be the $k$\textsuperscript{th} data bin, ${\rm RS}_{{\rm obs}}(k)$ (respectively $\ra_{{\rm obs}}(k)$) be the observed number of games with runs scored (allowed) in Bin($k$), and $A(\alpha,\beta, \gamma, k)$ be the area under the Weibull distribution with parameters $(\alpha, \beta, \gamma)$ in Bin($k$). Then for each team we are searching for the values of $(\alpha_{\rs}, \alpha_{\ra}, \gamma)$ that minimize
\begin{eqnarray}
& & \sum_{k=1}^{12} \left({\rm RS}_{{\rm obs}}(k) - 162 \cdot A(\alpha_{\rs}, -.5, \gamma, k)\right)^2  \nonumber\\ & & \ \ \ \ \
+ \ \sum_{k=1}^{12} \left({\rm RA}_{{\rm obs}}(k) - 162 \cdot A(\alpha_{\ra}, -.5, \gamma, k)\right)^2
\end{eqnarray} (the 162 is because the teams play 162 games in a season; if a team has fewer games, either due to a cancelled game or because we are analyzing another sport, this number is trivially adjusted).

For each team we found the best Weibulls with parameters $(\alpha_{\rs}, -.5, \gamma)$ and $(\alpha_{\ra}, -.5, \gamma)$ and then compared
the number of wins, losses, and won-loss percentage predicted by our model with the recorded data. The results are summarized in Table \ref{table:pythagleastsquaresWDiffGamesGamma}.

%%%%%% LEAST SQUARES PREDICTIONS GRAPHIC HERE%%%%%%
\begin{table}
\begin{center}
\begin{tabular}{lrrrrrrrrrrrr}
\hline
 {\rm Team} & \ \ \ &   {\rm Obs} {\rm W} & \ \ \ &   {\rm Pred} {\rm W} & \ \ \ &   {\rm Obs \%} & \ \ \ &   {\rm Pred \%} & \ \ \ &   {\rm Diff} {\rm Games} & \ \ \ &   {\rm      $\gamma $  } \\ \hline
 {\rm Washington Nationals} & \ \ \ &   98 & \ \ \ &   97.5 & \ \ \ &   0.605 & \ \ \ &   0.602 & \ \ \ &   0.5 & \ \ \ &   1.76 \\
 {\rm Cincinnati Reds} & \ \ \ &   97 & \ \ \ &   90.7 & \ \ \ &   0.599 & \ \ \ &   0.560 & \ \ \ &   6.3 & \ \ \ &   1.80 \\
 {\rm New York Yankees} & \ \ \ &   95 & \ \ \ &   96.0 & \ \ \ &   0.586 & \ \ \ &   0.593 & \ \ \ &   -1.0 & \ \ \ &   1.95 \\
 {\rm Oakland Athletics} & \ \ \ &   94 & \ \ \ &   89.8 & \ \ \ &   0.580 & \ \ \ &   0.554 & \ \ \ &   4.2 & \ \ \ &   1.54 \\
 {\rm San Francisco Giants} & \ \ \ &   94 & \ \ \ &   86.1 & \ \ \ &   0.580 & \ \ \ &   0.531 & \ \ \ &   7.9 & \ \ \ &   1.72 \\
 {\rm Atlanta Braves} & \ \ \ &   94 & \ \ \ &   89.4 & \ \ \ &   0.580 & \ \ \ &   0.552 & \ \ \ &   4.6 & \ \ \ &   1.51 \\
 {\rm Texas Rangers} & \ \ \ &   93 & \ \ \ &   91.0 & \ \ \ &   0.574 & \ \ \ &   0.562 & \ \ \ &   2.0 & \ \ \ &   1.69 \\
 {\rm Baltimore Orioles} & \ \ \ &   93 & \ \ \ &   83.1 & \ \ \ &   0.574 & \ \ \ &   0.513 & \ \ \ &   9.9 & \ \ \ &   1.66 \\
 {\rm Tampa Bay Rays} & \ \ \ &   90 & \ \ \ &   90.9 & \ \ \ &   0.556 & \ \ \ &   0.561 & \ \ \ &   -0.9 & \ \ \ &   1.75 \\
 {\rm Los Angeles Angels} & \ \ \ &   89 & \ \ \ &   86.4 & \ \ \ &   0.549 & \ \ \ &   0.533 & \ \ \ &   2.6 & \ \ \ &   1.59 \\
 {\rm Detroit Tigers} & \ \ \ &   88 & \ \ \ &   94.7 & \ \ \ &   0.543 & \ \ \ &   0.585 & \ \ \ &   -6.7 & \ \ \ &   1.89 \\
 {\rm St. Louis Cardinals} & \ \ \ &   88 & \ \ \ &   91.0 & \ \ \ &   0.543 & \ \ \ &   0.562 & \ \ \ &   -3.0 & \ \ \ &   1.66 \\
 {\rm Los Angeles Dodgers} & \ \ \ &   86 & \ \ \ &   87.9 & \ \ \ &   0.531 & \ \ \ &   0.542 & \ \ \ &   -1.9 & \ \ \ &   1.65 \\
 {\rm Chicago White Sox} & \ \ \ &   85 & \ \ \ &   87.1 & \ \ \ &   0.525 & \ \ \ &   0.538 & \ \ \ &   -2.1 & \ \ \ &   1.66 \\
 {\rm Milwaukee Brewers} & \ \ \ &   83 & \ \ \ &   85.0 & \ \ \ &   0.512 & \ \ \ &   0.525 & \ \ \ &   -2.0 & \ \ \ &   1.75 \\
 {\rm Philadelphia Phillies} & \ \ \ &   81 & \ \ \ &   76.7 & \ \ \ &   0.500 & \ \ \ &   0.474 & \ \ \ &   4.3 & \ \ \ &   1.72 \\
 {\rm Arizona Diamondbacks} & \ \ \ &   81 & \ \ \ &   84.8 & \ \ \ &   0.500 & \ \ \ &   0.524 & \ \ \ &   -3.8 & \ \ \ &   1.61 \\
 {\rm Pittsburgh Pirates} & \ \ \ &   79 & \ \ \ &   80.3 & \ \ \ &   0.488 & \ \ \ &   0.496 & \ \ \ &   -1.3 & \ \ \ &   1.63 \\
 {\rm San Diego Padres} & \ \ \ &   76 & \ \ \ &   74.7 & \ \ \ &   0.469 & \ \ \ &   0.461 & \ \ \ &   1.3 & \ \ \ &   1.65 \\
 {\rm Seattle Mariners} & \ \ \ &   75 & \ \ \ &   74.6 & \ \ \ &   0.463 & \ \ \ &   0.461 & \ \ \ &   0.4 & \ \ \ &   1.59 \\
 {\rm New York Mets} & \ \ \ &   74 & \ \ \ &   75.7 & \ \ \ &   0.457 & \ \ \ &   0.467 & \ \ \ &   -1.7 & \ \ \ &   1.63 \\
 {\rm Toronto Blue Jays} & \ \ \ &   73 & \ \ \ &   73.7 & \ \ \ &   0.451 & \ \ \ &   0.455 & \ \ \ &   -0.7 & \ \ \ &   1.66 \\
 {\rm Kansas City Royals} & \ \ \ &   72 & \ \ \ &   74.8 & \ \ \ &   0.444 & \ \ \ &   0.462 & \ \ \ &   -2.8 & \ \ \ &   1.78 \\
 {\rm Boston Red Sox} & \ \ \ &   69 & \ \ \ &   73.6 & \ \ \ &   0.426 & \ \ \ &   0.455 & \ \ \ &   -4.6 & \ \ \ &   1.72 \\
 {\rm Miami Marlins} & \ \ \ &   69 & \ \ \ &   76.1 & \ \ \ &   0.426 & \ \ \ &   0.470 & \ \ \ &   -7.1 & \ \ \ &   1.74 \\
 {\rm Cleveland Indians} & \ \ \ &   68 & \ \ \ &   65.2 & \ \ \ &   0.420 & \ \ \ &   0.402 & \ \ \ &   2.8 & \ \ \ &   1.76 \\
 {\rm Minnesota Twins} & \ \ \ &   66 & \ \ \ &   65.8 & \ \ \ &   0.407 & \ \ \ &   0.406 & \ \ \ &   0.2 & \ \ \ &   1.91 \\
 {\rm Colorado Rockies} & \ \ \ &   64 & \ \ \ &   71.0 & \ \ \ &   0.395 & \ \ \ &   0.438 & \ \ \ &   -7.0 & \ \ \ &   1.79 \\
 {\rm Chicago Cubs} & \ \ \ &   61 & \ \ \ &   70.6 & \ \ \ &   0.377 & \ \ \ &   0.436 & \ \ \ &   -9.6 & \ \ \ &   1.58 \\
 {\rm Houston Astros} & \ \ \ &   55 & \ \ \ &   61.3 & \ \ \ &   0.340 & \ \ \ &   0.379 & \ \ \ &   -6.3 & \ \ \ &   1.61 \\
\hline
\end{tabular}
\caption{\label{table:pythagleastsquaresWDiffGamesGamma} Results from best fit values from the Method of Least Squares, displaying the observed and predicted number of wins, winning percentage, and difference in games won and predicted for the 2012 season.}
\end{center}
\end{table}

%\begin{center}
%\begin{figure}[h]
%%%%%Method of Least Squares Predictions for 2012 Season
%\includegraphics[scale=.7]{pythag2012_predictionswgamma.eps}
%\caption{\label{figure:pythagmethodlastsquares2012} Best fit values and quality of fit from the Method of Least Squares for the 2012 season. The %first two columns are observed and predicted wins, followed by observed and predicted win percentage, the difference in wins between the observed %and predicted values, and the best fit value of $\gamma$.}
%\end{figure}
%\end{center}
%%%%%% TEXT FOR LEAST SQUARES PREDICTIONS GRAPHIC %%%%%%%

The mean of $\gamma$ over the 30 teams for the 2012 season is 1.70 with a standard deviation of .11. This is slightly lower than the value in the literature of 1.82. The difference between the two methods is that our value of $\gamma$ is a consequence of our model, whereas the 1.82 comes from assuming the Pythagorean formula is valid and finding which exponent gives the best fit to the observed winning percentages. We discuss ways to improve our model in \S\ref{sec:pythaggeneralization}.

Comparing the predicted number of wins with the observed number of wins, we see that the mean difference between these quantities is about -.52 with a standard deviation of about 4.61. This data is misleading, though, as the mean difference is small as these are signed quantities. It is thus better to examine the absolute value of the difference between observed and predicted wins. Doing so gives an average value of about 3.65 with a standard deviation around 2.79, consistent with the empirical result that the Pythagorean formula is usually accurate to around four wins a season.

We next examine each team's $z$-score for the difference between the observed and predicted runs scored and runs allowed. A $z$-test is appropriate
here because of the large number of games played by each team, a crucial difference between baseball and football. The critical value corresponding to a 95\% confidence level is 1.96, while the value for the 99\% level is 2.575. The $z$-score (for runs scored) for a given team is defined as follows. Let $\rs_{{\rm obs}}$ denote the observed average runs scored, $\rs_{{\rm pred}}$ the predicted average runs scored (from the best fit Weibull), $\sigma_{{\rm obs}}$ the standard deviation of the observed runs scored, and remember there are 162 games in a season. Then \begin{equation} z_\rs \ = \ \frac{\rs_{{\rm obs}} - \rs_{{\rm pred}}}{\sigma_{{\rm obs}} / \sqrt{162}}. \end{equation} We see in Table \ref{table:pythagleastsquaresztests} that both the runs scored and runs allowed $z$-statistics almost always fall well below 1.96 in absolute value, indicating that the parameters estimated by the Method of Least Squares predict the observed data well. We could do a Bonferroni adjustment for multiple comparisons as these are not independent comparisons, which allows us to divide the confidence levels by 30 (the number of comparisons); this is a very conservative statistic. Doing so increases the thresholds to approximately 2.92 and 3.38, to the point that all values are in excellent agreement with theory.

\begin{table}
\begin{center}
\begin{tabular}{lrrrrrrrrrrrr}
\hline
 {\rm Team} & \ \ \ &  {\rm Obs RS} & \ \ \ &  {\rm Pred RS} & \ \ \ &  {\rm $z$-stat} & \ \ \ &  {\rm Obs RA} & \ \ \ &  {\rm Pred RA} & \ \ \ &  {\rm $z$-stat} \\
 {\rm Washington Nationals} & \ \ \ &  4.51 & \ \ \ &  4.54 & \ \ \ &  -0.13 & \ \ \ &  3.67 & \ \ \ &  3.49 & \ \ \ &  0.87 \\
 {\rm Cincinnati Reds} & \ \ \ &  4.13 & \ \ \ &  4.13 & \ \ \ &  {\rm   0.00} & \ \ \ &  3.63 & \ \ \ &  3.55 & \ \ \ &  0.39 \\
 {\rm New York Yankees} & \ \ \ &  4.96 & \ \ \ &  5.02 & \ \ \ &  -0.24 & \ \ \ &  4.12 & \ \ \ &  4.05 & \ \ \ &  0.33 \\
 {\rm Oakland Athletics} & \ \ \ &  4.40 & \ \ \ &  4.48 & \ \ \ &  -0.30 & \ \ \ &  3.79 & \ \ \ &  3.82 & \ \ \ &  -0.15 \\
 {\rm San Francisco Giants} & \ \ \ &  4.43 & \ \ \ &  4.36 & \ \ \ &  0.32 & \ \ \ &  4.01 & \ \ \ &  4.02 & \ \ \ &  -0.05 \\
 {\rm Atlanta Braves} & \ \ \ &  4.32 & \ \ \ &  4.39 & \ \ \ &  -0.27 & \ \ \ &  3.70 & \ \ \ &  3.76 & \ \ \ &  -0.27 \\
 {\rm Texas Rangers} & \ \ \ &  4.99 & \ \ \ &  4.86 & \ \ \ &  0.48 & \ \ \ &  4.36 & \ \ \ &  4.13 & \ \ \ &  0.88 \\
 {\rm Baltimore Orioles} & \ \ \ &  4.40 & \ \ \ &  4.41 & \ \ \ &  -0.09 & \ \ \ &  4.35 & \ \ \ &  4.26 & \ \ \ &  0.35 \\
 {\rm Tampa Bay Rays} & \ \ \ &  4.30 & \ \ \ &  4.18 & \ \ \ &  0.52 & \ \ \ &  3.56 & \ \ \ &  3.57 & \ \ \ &  -0.04 \\
 {\rm Los Angeles Angels} & \ \ \ &  4.73 & \ \ \ &  4.84 & \ \ \ &  -0.42 & \ \ \ &  4.31 & \ \ \ &  4.41 & \ \ \ &  -0.38 \\
 {\rm Detroit Tigers} & \ \ \ &  4.48 & \ \ \ &  4.49 & \ \ \ &  -0.03 & \ \ \ &  4.14 & \ \ \ &  3.66 & \ \ \ &  2.03 \\
 {\rm St. Louis Cardinals} & \ \ \ &  4.72 & \ \ \ &  4.73 & \ \ \ &  -0.05 & \ \ \ &  4.00 & \ \ \ &  4.01 & \ \ \ &  -0.02 \\
 {\rm Los Angeles Dodgers} & \ \ \ &  3.93 & \ \ \ &  4.07 & \ \ \ &  -0.67 & \ \ \ &  3.69 & \ \ \ &  3.63 & \ \ \ &  0.29 \\
 {\rm Chicago White Sox} & \ \ \ &  4.62 & \ \ \ &  4.60 & \ \ \ &  0.09 & \ \ \ &  4.17 & \ \ \ &  4.15 & \ \ \ &  0.09 \\
 {\rm Milwaukee Brewers} & \ \ \ &  4.79 & \ \ \ &  4.89 & \ \ \ &  -0.41 & \ \ \ &  4.52 & \ \ \ &  4.59 & \ \ \ &  -0.30 \\
 {\rm Philadelphia Phillies} & \ \ \ &  4.22 & \ \ \ &  4.08 & \ \ \ &  0.61 & \ \ \ &  4.20 & \ \ \ &  4.37 & \ \ \ &  -0.82 \\
 {\rm Arizona Diamondbacks} & \ \ \ &  4.53 & \ \ \ &  4.59 & \ \ \ &  -0.24 & \ \ \ &  4.25 & \ \ \ &  4.30 & \ \ \ &  -0.26 \\
 {\rm Pittsburgh Pirates} & \ \ \ &  4.02 & \ \ \ &  4.12 & \ \ \ &  -0.45 & \ \ \ &  4.16 & \ \ \ &  4.17 & \ \ \ &  -0.04 \\
 {\rm San Diego Padres} & \ \ \ &  4.02 & \ \ \ &  4.09 & \ \ \ &  -0.35 & \ \ \ &  4.38 & \ \ \ &  4.55 & \ \ \ &  -0.76 \\
 {\rm Seattle Mariners} & \ \ \ &  3.82 & \ \ \ &  3.68 & \ \ \ &  0.60 & \ \ \ &  4.02 & \ \ \ &  4.11 & \ \ \ &  -0.44 \\
 {\rm New York Mets} & \ \ \ &  4.01 & \ \ \ &  4.06 & \ \ \ &  -0.24 & \ \ \ &  4.38 & \ \ \ &  4.44 & \ \ \ &  -0.26 \\
 {\rm Toronto Blue Jays} & \ \ \ &  4.42 & \ \ \ &  4.37 & \ \ \ &  0.19 & \ \ \ &  4.84 & \ \ \ &  4.93 & \ \ \ &  -0.35 \\
 {\rm Kansas City Royals} & \ \ \ &  4.17 & \ \ \ &  4.21 & \ \ \ &  -0.17 & \ \ \ &  4.60 & \ \ \ &  4.63 & \ \ \ &  -0.09 \\
 {\rm Boston Red Sox} & \ \ \ &  4.53 & \ \ \ &  4.33 & \ \ \ &  0.79 & \ \ \ &  4.98 & \ \ \ &  4.87 & \ \ \ &  0.40 \\
 {\rm Miami Marlins} & \ \ \ &  3.76 & \ \ \ &  3.96 & \ \ \ &  -0.96 & \ \ \ &  4.47 & \ \ \ &  4.29 & \ \ \ &  0.80 \\
 {\rm Cleveland Indians} & \ \ \ &  4.12 & \ \ \ &  4.06 & \ \ \ &  0.22 & \ \ \ &  5.22 & \ \ \ &  5.21 & \ \ \ &  {\rm   0.00} \\
 {\rm Minnesota Twins} & \ \ \ &  4.33 & \ \ \ &  4.14 & \ \ \ &  0.71 & \ \ \ &  5.14 & \ \ \ &  5.16 & \ \ \ &  -0.12 \\
 {\rm Colorado Rockies} & \ \ \ &  4.68 & \ \ \ &  4.75 & \ \ \ &  -0.29 & \ \ \ &  5.49 & \ \ \ &  5.53 & \ \ \ &  -0.16 \\
 {\rm Chicago Cubs} & \ \ \ &  3.78 & \ \ \ &  3.89 & \ \ \ &  -0.50 & \ \ \ &  4.69 & \ \ \ &  4.67 & \ \ \ &  0.05 \\
 {\rm Houston Astros} & \ \ \ &  3.60 & \ \ \ &  3.57 & \ \ \ &  0.13 & \ \ \ &  4.90 & \ \ \ &  5.04 & \ \ \ &  -0.57 \\
\hline
\end{tabular}
\caption{\label{table:pythagleastsquaresztests} Method of Least Squares: $z$-tests for best fit runs scored and allowed.}
\end{center}
\end{table}

To further demonstrate the quality of the fit, in Figure \ref{figure:pythagpiratesRSRA} we compare the best fit Weibulls with the Pittsburgh Pirates (who were essentially a .500 team and thus in the middle of the pack). The fit is excellent .

\begin{figure}[h]
\begin{center}
\scalebox{.62}{\includegraphics{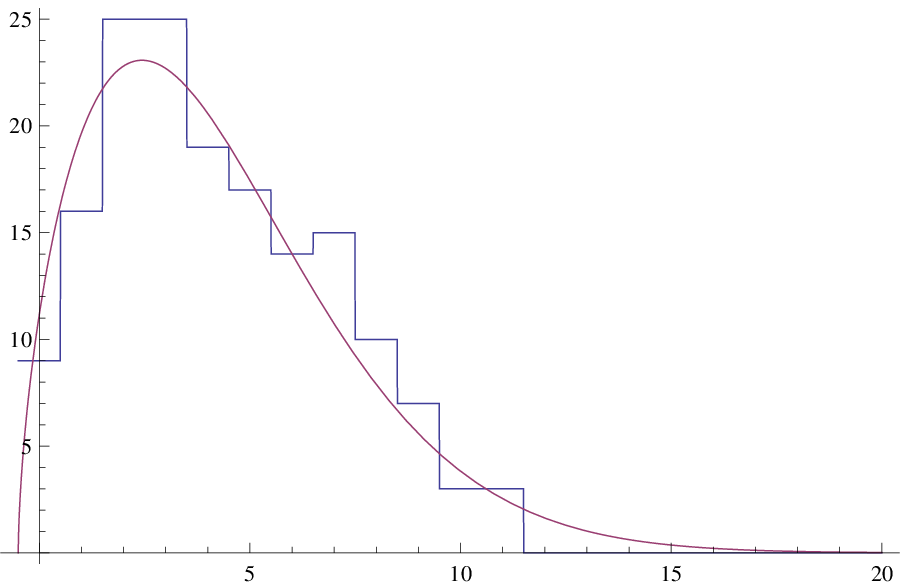}}\ \scalebox{.62}{\includegraphics{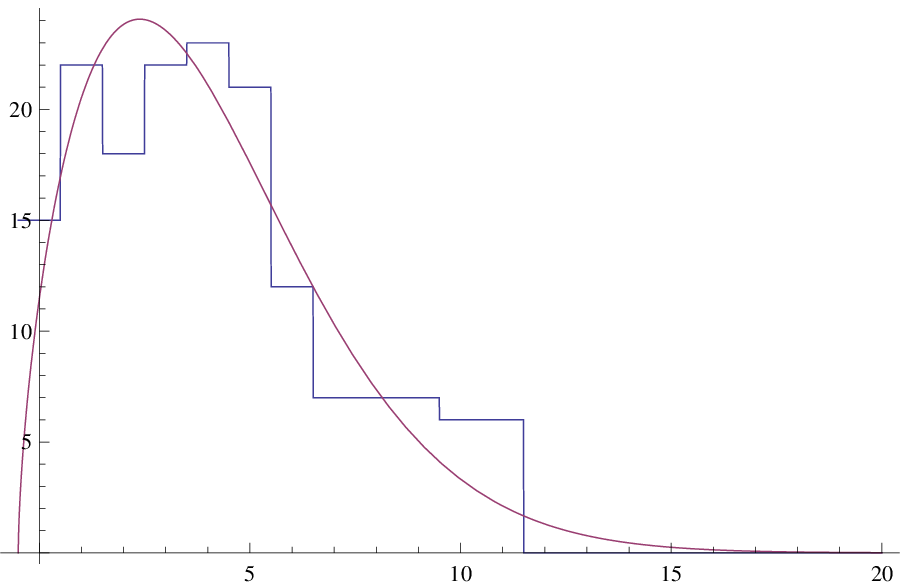}}
\caption{\label{figure:pythagpiratesRSRA} Comparison of the best fit Weibulls for runs scored (left) and allowed (right) for the 2012 Pittsburgh Pirates against the observed distribution of scores.}
\end{center}\end{figure}

%\begin{center}
%\begin{figure}[h]
%%%%%Method of Least Squares Analysis of Parameter Estimation
%\includegraphics[scale=.7]{2012_zstats.eps}
%\caption{\label{figure:pythagMLSztests} Method of Least Squares: $z$-tests for best fit runs scored and allowed.}
%\end{figure}
%\end{center}

We now come to the most important part of the analysis, testing the assumptions that the runs scored and allowed are given by independent Weibulls. We do this in two stages. We first see how well the Weibulls do fitting the data, and whether or not the runs scored and allowed are statistically independent (other than the restriction that they are not equal). We describe the analysis first, and then present the results in Table \ref{table:pythagleastsquaresRSRAindependence}. As the independence test is complicated by the presence of structural zeros (unattainable values), we provide a detailed description here for the benefit of the reader.

The first column in Table \ref{table:pythagleastsquaresRSRAindependence} is a $\chi^2$ goodness of fit test to determine how closely the observed data follows a Weibull distribution with the estimated parameters, using the same bins as before. Our test statistic is
\begin{eqnarray}
&  & \sum_{k=1}^{12} \frac{\left(\rs_{{\rm obs}}(k) - 162 \cdot A(\alpha_{\rs}, -.5, \gamma, k)\right)^2}
{162 \cdot A(\alpha_{\rs}, -.5, \gamma, k)} \nonumber\\ & & \ \ \ \ \ +\
\sum_{k=1}^{12} \frac{\left(\ra_{{\rm obs}}(k) - 162 \cdot A(\alpha_{\ra}, -.5, \gamma, k)\right)^2}
{162 \cdot A(\alpha_{\ra}, -.5, \gamma, k)}.
\end{eqnarray}
This test has 20 degrees of freedom, which corresponds to critical values of 31.41 (95\% level) and 37.57 (99\% level). Of course, as we have multiple comparisons we should again perform a Bonferroni adjustment. We divide the significance levels by 30, the number of comparisons, and thus the values increase to 43.67 and 48.75. Almost all the teams are now in range, with the only major outliers being the Yankees and the Rays, the two playoff teams from the American League East.

We now turn to the final key assumption, the independence of runs scored and runs allowed, by doing a $\chi^2$ test for
independence. This test involves creating a contingency table with the requirement that each row and column has at least one non-zero entry.
As the Miami Marlins had no games with 10 runs scored, we had to slightly modify our choice of bins to
\begin{equation}
[0,1), \ \ [1, 2), \ \ \dots, \ \ [9,11), \ \ [11, \infty);
\end{equation} as we are using the observed run data from games, we can have our bins with left endpoints at the integers.

We have an 11 $\times$ 11 contingency table. As runs scored cannot equal runs allowed in a game (games cannot end in a tie), we are forced to have zeroes along the diagonal. The constraint on the values of runs scored and runs allowed leads to an incomplete two-dimensional contingency table with $(11-1)^2 - 11 = 89$ degrees of freedom. We briefly review the theory of such tests with structural zeros in Appendix \ref{sec:pythagappendixindstructuralzeros}. The critical values for a $\chi^2$ test with 89 degrees of freedom are 113.15 (95\% level) and 124.12 (99\% level). Table \ref{table:pythagleastsquaresRSRAindependence} shows that all chi-square values for the teams in the 2012 season fall below the 99\% level, indicating that runs scored and runs allowed are behaving as if they are statistically independent. The fits are even better if we use the Bonferroni adjustments, which are 133.26 and 141.56.

\begin{table}
\begin{center}
\begin{tabular}{lrrrr}
\hline
 {\rm Team} & \ \ \ &  {\rm RS+RA}\ $\chi^2$: 20 d.f.  & \ \ \ &  {\rm Independence} $\chi^2$: 109 d.f  \\ \hline
 {\rm Washington Nationals} & \ \ \ &  53.80 & \ \ \ &  101.07 \\
 {\rm Cincinnati Reds} & \ \ \ &  33.69 & \ \ \ &  107.11 \\
 {\rm New York Yankees} & \ \ \ &  64.02 & \ \ \ &  82.82 \\
 {\rm Oakland Athletics} & \ \ \ &  22.34 & \ \ \ &  87.85 \\
 {\rm San Francisco Giants} & \ \ \ &  14.37 & \ \ \ &  89.57 \\
 {\rm Atlanta Braves} & \ \ \ &  32.34 & \ \ \ &  101.07 \\
 {\rm Texas Rangers} & \ \ \ &  26.49 & \ \ \ &  93.46 \\
 {\rm Baltimore Orioles} & \ \ \ &  11.90 & \ \ \ &  98.29 \\
 {\rm Tampa Bay Rays} & \ \ \ &  66.35 & \ \ \ &  120.25 \\
 {\rm Los Angeles Angels} & \ \ \ &  28.10 & \ \ \ &  105.73 \\
 {\rm  Detroit Tigers} & \ \ \ &  38.76 & \ \ \ &  98.96 \\
 {\rm St. Louis Cardinals} & \ \ \ &  36.32 & \ \ \ &  117.21 \\
 {\rm Los Angeles Dodgers} & \ \ \ &  31.70 & \ \ \ &  123.33 \\
 {\rm Chicago White Sox} & \ \ \ &  20.61 & \ \ \ &  121.33 \\
 {\rm Milwaukee Brewers} & \ \ \ &  49.51 & \ \ \ &  98.02 \\
 {\rm Philadelphia Phillies} & \ \ \ &  19.19 & \ \ \ &  93.78 \\
 {\rm Arizona Diamondbacks} & \ \ \ &  23.91 & \ \ \ &  78.44 \\
 {\rm Pittsburgh Pirates} & \ \ \ &  13.46 & \ \ \ &  103.85 \\
 {\rm San Diego Padres} & \ \ \ &  17.62 & \ \ \ &  92.87 \\
 {\rm Seattle Mariners} & \ \ \ &  9.79 & \ \ \ &  113.13 \\
 {\rm New York Mets} & \ \ \ &  42.88 & \ \ \ &  95.66 \\
 {\rm Toronto Blue Jays} & \ \ \ &  13.09 & \ \ \ &  86.81 \\
 {\rm Kansas City Royals} & \ \ \ &  22.51 & \ \ \ &  102.39 \\
 {\rm Boston Red Sox} & \ \ \ &  22.43 & \ \ \ &  99.18 \\
 {\rm Miami Marlins} & \ \ \ &  43.64 & \ \ \ &  121.32 \\
 {\rm Cleveland Indians} & \ \ \ &  26.62 & \ \ \ &  83.28 \\
 {\rm Minnesota Twins} & \ \ \ &  50.40 & \ \ \ &  115.04 \\
 {\rm Colorado Rockies} & \ \ \ &  24.30 & \ \ \ &  85.79 \\
 {\rm Chicago Cubs} & \ \ \ &  40.06 & \ \ \ &  90.72 \\
 {\rm Houston Astros} & \ \ \ &  41.16 & \ \ \ &  80.48 \\
\hline
\end{tabular}
\caption{\label{table:pythagleastsquaresRSRAindependence} Results from best fit values from the Method of Least Squares for 2012, displaying the quality of the fit of the Weibulls to the observed scoring data, and testing the independence of runs scored and allowed.}
\end{center}
\end{table}

%\begin{center}
%\begin{figure}[h]
%\includegraphics[scale=.9]{2012_chisqindep.eps}
%\caption{\label{figure:pythagweibullsindependence} Chi-squared test for goodness of the fit of the Wiebulls and the independence of runs scored and %allowed. The critical values for 95\% and 99\% are 31.41 and 37.57 for 20 degrees of freedom, and 113.15 and 124.12 for 89 degrees of freedom (the %Bonferroni adjustments increase the values to 43.67 and 48.75, and 133.26 and 141.56).}
%\end{figure}
%\end{center}

\subsection{Analysis of $\gamma$ and Games Off}

Given a dataset and a statistical model, the method of maximum likelihood is a technique that computes the parameters of the model that make the observed data most probable. Maximum likelihood estimators have the desirable property of being asymptotically minimum variance unbiased estimators. Based on the statistical model in question, one constructs the likelihood function. For our model, if we have $B$ bins then the likelihood function is given by
\begin{eqnarray}
& &  L(\alpha_\rs, \alpha_\ra, -.5,\gamma)\ = \ \ncr{162}{\rso(1), \dots, \rso(B)} \prod_{k=1}^{B} A(\alpha_\rs,-.5,\gamma,k)^{\rso(k)}
\nonumber\\ & & \ \ \ \ \ \ \cdot \ \ncr{162}{\rao(1), \dots,
\rao(B)} \prod_{k=1}^{B}
A(\alpha_\ra,-.5,\gamma,k)^{\rao(k)}. \end{eqnarray}

The maximum likelihood estimators are found by determining the values of the parameters $\alpha_\rs$, $\alpha_\ra$ and $\gamma$  that maximize the likelihood function. In practice one typically maximizes the logarithm of the likelihood because it is both equivalent to and computationally easier than maximizing the likelihood function directly.

Using our model, we calculated the maximum likelihood estimators for each team. Figure \ref{figure:pythagmaxlikegammaplots} displays the average values of the parameter $\gamma$ for each season from 2007 to 2012, with error bars indicating the standard deviation. Note that the standard deviation of the $\gamma$ values for each season are similar to each other, with 2010 having the largest deviation. The mean value of $\gamma$ is about 1.69 with a standard deviation of .03.

\begin{figure}[h]
\begin{center}
\includegraphics[scale=1.0]{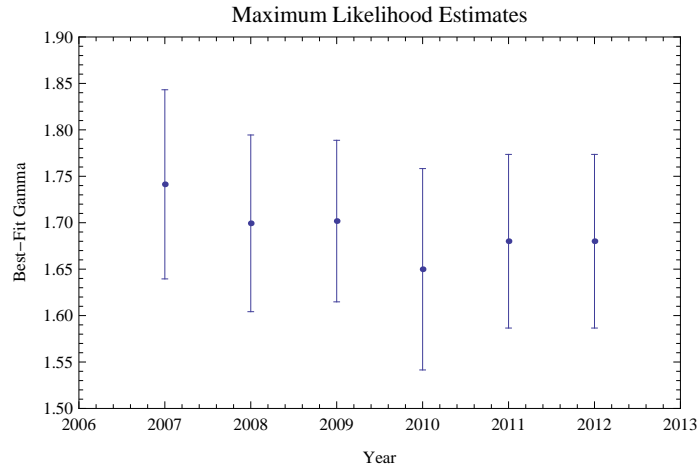}
\caption{\label{figure:pythagmaxlikegammaplots} Average value of $\gamma$ from the Method of Maximum Likelihood.}
\end{center}
\end{figure}

Using the maximum likelihood estimators, we then calculated the predicted number of games won for each team and compared this to the observed numbers. The average absolute value of this difference is shown for each year in Figure \ref{figure:pythagmaxlikegamesoffplots}, with error bars indicating the standard deviation. The mean of the absolute value of the games off by is approximately 3.81, with a standard deviation of about .94; these numbers are in-line with the conventional wisdom that the Pythagorean formula is typically accurate to about 4 games per season.

\begin{figure}[h]
\begin{center}
\includegraphics[scale=1.0]{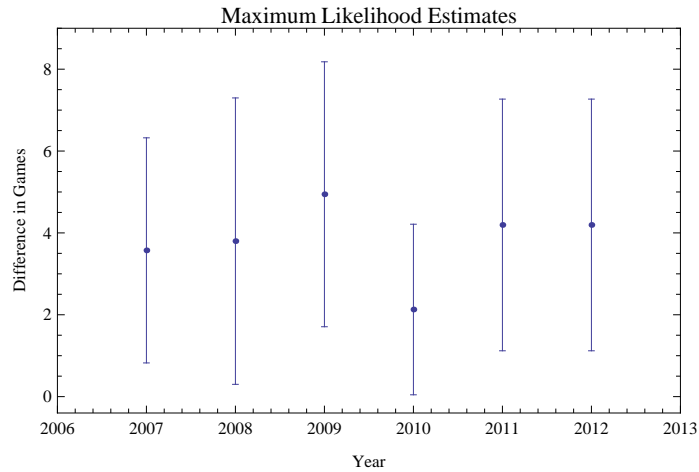}
\caption{\label{figure:pythagmaxlikegamesoffplots} Average absolute value of the difference between the observed and predicted number of wins from the Method of Maximum Likelihood.}
\end{center}
\end{figure}

%%%%%%%%%%%%%%%%%%%%%%%%%%%%%%%%%%%%%%%%%%%%%%%%%%%%%%%%%%%%%%%%%%%%%%%%%%%%%%%%%%%%%%%%%%%%%%%%%%%%%%%%%%%%%%%%%%%%%%%%%%%%%%%%%%%%%%%%%%%%%%%
%%%%%%%%%%%%%%%%%%%%%%%%%%%%%%%%%%%%%%%%%%%%%%%%%%%%%%%%%%%%%%%%%%%%%%%%%%%%%%%%%%%%%%%%%%%%%%%%%%%%%%%%%%%%%%%%%%%%%%%%%%%%%%%%%%%%%%%%%%%%%%%
%%%%%%%%%%%%%%%%%%%%%%%%%%%%%%%%%%%%%%%%%%%%%%%%%%%%%%%%%%%%%%%%%%%%%%%%%%%%%%%%%%%%%%%%%%%%%%%%%%%%%%%%%%%%%%%%%%%%%%%%%%%%%%%%%%%%%%%%%%%%%%%

\section{The Pythagorean Formula: Linearization}\label{sec:pythaglinearization}

The Pythagorean formula is not the only predictor used, though it is one of the earliest and most famous. A popular alternative is a linear statistic. For example, Michael Jones and Linda Tappin [7] state that a good estimate for a team's winning percentage is $.500 + {\rm B} (\rs - \ra)$, where $\rs$ and $\ra$ are runs scored and allowed, and ${\rm B}$ is a small positive constant whose average in their studies was around 0.00065. \emph{Note here there is a difference if we use total runs or average runs per game, as we no longer have a ratio.} We can of course use average runs per game, but that would require rescaling ${\rm B}$; thus, for the rest of this section, we work in total runs.

While their formula is simpler to use, computers are handling all the calculations anyway and thus the savings over the Pythagorean formula is not significant. Further, by applying a Taylor series expansion to the Pythagorean formula we obtain not only this linear predictor, but also find an interpretation of ${\rm B}$ in terms of $\gamma$ and the average runs scored by teams. We give a simple proof using multivariable calculus; see Appendix \ref{sec:pythaglinearizingpythag} for an alternative proof that only requires one variable calculus. The multivariable argument was first given in [3] by Steven J. Miller and Kevin Dayaratna; the one-dimensional argument is from an unpublished appendix.

Given a multivariable function $f(x,y)$, if $(x,y)$ is close to $(a,b)$ then $f(x,y)$ is approximately the first order Taylor series about the point $(a,b)$: \begin{equation} f(a,b) + \frac{\partial f}{\partial x}(a,b) (x - a) + \frac{\partial f}{\partial y}(a,b) (y-b). \end{equation} We take \begin{equation} f(x,y)  \ = \ \frac{x^\gamma}{x^\gamma + y^\gamma}, \ \ \ \ (a,b) \ = \ (\rt, \rt), \end{equation} where $\rt$ is the average of the total runs scored in the league. After some algebra we find \begin{equation} \frac{\partial f}{\partial x}(x,y) \ = \ \frac{\gamma x^{\gamma-1} y^\gamma}{(x^\gamma+y^\gamma)^2}, \ \ \ \frac{\partial f}{\partial x}(\rt,\rt) \ = \ \frac{\gamma}{4\rt}, \end{equation} which is also $-\frac{\partial f}{\partial y}(\rt,\rt)$. Taking $(x,y) = (\rs, \ra)$, the first order Taylor series expansion becomes \begin{eqnarray} & &  f(\rt, \rt) + \frac{\gamma}{4\rt} (\rs - \rt) - \frac{\gamma}{4\rt} (\ra - \rt)\nonumber\\ &  & \ \ \ \ = \ .500 + \frac{\gamma}{4\rt} (\rs - \ra).\end{eqnarray}

Thus, not only do we obtain a linear estimator, but we have a theoretical prediction for the all-important slope ${\rm B}$, namely that ${\rm B} = \gamma / (4\rt)$. See the paper by Dayaratna and Miller [3] for a detailed analysis of how well this ratio fits ${\rm B}$. We content ourselves here with remarking that in 2012 the two leagues combined to score 21,017 runs (see \textcolor{blue}{\url{http://www.baseball-almanac.com/hitting/hiruns4.shtml}}), for an average of 4.32449 runs per game per team, or an average of 700.567 runs per team. Using 1.83 for $\gamma$ and 700.567 for $\rt$, we predict ${\rm B}$ should be about 0.000653, agreeing beautifully with Jones and Tappin's findings (see \textcolor{blue}{\url{http://www.sciencedaily.com/releases/2004/03/040330090259.htm}}).

%%%%%%%%%%%%%%%%%%%%%%%%%%%%%%%%%%%%%%%%%%%%%%%%%%%%%%%%%%%%%%%%%%%%%%%%%%%%%%%%%%%%%%%%%%%%%%%%%%%%%%%%%%%%%%%%%%%%%%%%%%%%%%%%%%%%%%%%%%%%%%%
%%%%%%%%%%%%%%%%%%%%%%%%%%%%%%%%%%%%%%%%%%%%%%%%%%%%%%%%%%%%%%%%%%%%%%%%%%%%%%%%%%%%%%%%%%%%%%%%%%%%%%%%%%%%%%%%%%%%%%%%%%%%%%%%%%%%%%%%%%%%%%%
%%%%%%%%%%%%%%%%%%%%%%%%%%%%%%%%%%%%%%%%%%%%%%%%%%%%%%%%%%%%%%%%%%%%%%%%%%%%%%%%%%%%%%%%%%%%%%%%%%%%%%%%%%%%%%%%%%%%%%%%%%%%%%%%%%%%%%%%%%%%%%%

\section{The Future of the Pythagorean Formula}\label{sec:pythaggeneralization}

In the last section we saw how to use calculus to linearize the Pythagorean formula and obtain simpler estimators. Of course, linearizing the Pythagorean formula is not the only extension (and, as we are throwing away information, it is clearly not the optimal choice). In current research, the author and his students are exploring more accurate models for teams. There are two disadvantages to this approach. The first is that the resulting formula will almost surely be more complicated than the current one, and the second is that more information will be required than the aggregate scoring.

These restrictions, however, are not severe. As computers are doing all the calculations anyway, it is preferable to have a more accurate formula at the cost of additional computations that will never be noticed. The second item is more severe. The formulas under development will not be computable from the information available on common standings pages, but instead will require \emph{inning by inning} data. Thus these statistics will not be computable by the layperson reading the sports page; however, this is true about most advanced statistics. For example, it is impossible to calculate the win probability added for a player without going through each moment of a game.

We therefore see that these additional requirements are perfectly fine for applications. Teams are concerned with making optimal decisions, and the new data required is readily available to them (and in many cases to the average fan who can write a script program to cull it from publicly available websites). The current expanded version of the Pythagorean formula, which is work in progress by the first and third authors of this paper [8], will include the following three ingredients, all of which are easily done with readily available data.

\begin{enumerate}

\item Write the distribution for runs scored and allowed as a linear combinations of Weibulls.

\item Adjust the value of a run scored and allowed based on the ballpark.

\item Discount runs scored and allowed from a team's statistics based on the game state.

\end{enumerate}

The reason runs scored and allowed are modeled by Weibulls is that these lead to tractable, closed form integration. We can still perform the integration if instead each distribution is replaced with a linear combination of Weibulls; this is similar in spirit to the multitude of weights that occur in numerous other statistics, and will lead to a weighted sum of Pythagorean expressions for the winning percentage. An additional topic to be explored is allowing for dependencies between runs scored and allowed, but this is significantly harder and almost surely will lead to non-closed form solutions. It is highly desirable to have a closed form solution, as then we can estimate the value of a player by substituting their contributions into the formula and avoid the need for intense simulations.

The second change is trivial and easily done; certain ballparks favor pitchers while others favor hitters. The difficulty in scoring a run at Fenway Park is not the same as scoring one in Yankee Stadium, and thus ballpark effects should be used to adjust the values of the runs.

Finally, anyone who has turned on the TV during election night knows that certain states are called quickly after polls closed; the preliminary poll data is enough to predict with incredible accuracy what will happen. If a team has a large lead late in the game, they often rest their starters or use weaker pitchers, and thus the runs scored and allowed data here is not as indicative of a team's ability as earlier in the game. For example, in 2005 Mike Remlinger was traded to the Red Sox. In his first two games he allowed 5 runs to score (2 earned) while recording no outs; his ERA for the season to date was 5.45 and his win probability added was slightly negative. On August 16 the Sox and the Tigers were tied after 9 due to an Ortiz home-run in the ninth.\footnote{The data below is from \textcolor{blue}{http://www.baseball-reference.com/players/gl.cgi?id=remlimi01\&t=p\&year=2005} and \textcolor{blue}{http://scores.espn.go.com/mlb/boxscore?gameId=250816106}.} Ortiz had a three run shot the following inning, part of a 7 run offensive at the start of the tenth. With a seven run lead, this should not have been a critical situation, and Remlinger entered the game to pitch the bottom of the tenth. After retiring the first two batters, two walks and an infield single later it was bases loaded. Monroe then homered to make it 10-7, but Remlinger rallied and retired Inge. There were two reasons Papelbon was not brought in for the tenth. The first is that back then Papelbon was a starter (and in fact started that game!). More importantly, however, with a 7 run lead and just one inning to play, the leverage of the situation was low. Thus it is inappropriate to treat all runs equally. This mistake occurs in other sports; for example, when the Pythagorean formula is applied in football practitioners frequently do not adjust for the fact that at the end of the season certain teams have already locked up their playoff seed and are resting starters.

The hope is that incorporating these and other modifications will result in a more accurate Pythagorean formula. Though it will not be as easy to use, it will still be computable with known data and not require any simulations, and almost surely provide a better evaluation of a player's worth to their team.

%%%%%%%%%%%%%%%%%%%%%%%%%%%%%%%%%%%%%%%%%%%%%%%%%%%%%%%%%%%%%%%%%%%%%%%%%%%%%%%%%%%%%%%%%%%%%%%%%%%%%%%%%%%%%%%%%%%%%%%%%%%%%%%%%%%%%%%%%%%%%%%
%%%%%%%%%%%%%%%%%%%%%%%%%%%%%%%%%%%%%%%%%%%%%%%%%%%%%%%%%%%%%%%%%%%%%%%%%%%%%%%%%%%%%%%%%%%%%%%%%%%%%%%%%%%%%%%%%%%%%%%%%%%%%%%%%%%%%%%%%%%%%%%
%%%%%%%%%%%%%%%%%%%%%%%%%%%%%%%%%%%%%%%%%%%%%%%%%%%%%%%%%%%%%%%%%%%%%%%%%%%%%%%%%%%%%%%%%%%%%%%%%%%%%%%%%%%%%%%%%%%%%%%%%%%%%%%%%%%%%%%%%%%%%%%

\begin{acknowledgement}
The first author was partially supported by NSF Grants DMS0970067 and DMS1265673. He thanks Chris Chiang for suggesting the title of this talk, numerous students of his at Brown University and Williams College, as well as Cameron and Kayla Miller, for many lively conversations on mathematics and sports, Michael Stone for comments on an earlier draft, and Phil Birnbaum, Kevin Dayaratna, Warren Johnson and Chris Long for many sabermetrics discussions. This paper is dedicated to his great uncle Newt Bromberg, who assured him he would live long enough to see the Red Sox win it all, and the 2004, 2007 and 2013 Red Sox who made it happen (after the 2013 victory his six year old son Cameron turned to him and commented that he got to see it at a much younger age!). \end{acknowledgement}
\section{Appendix}
\addcontentsline{toc}{section}{Appendix}
%
%
%When placed at the end of a chapter or contribution (as opposed to at the end of the book), the numbering of tables, figures, and equations in the appendix section continues on from that in the main text. Hence please \textit{do not} use the \verb|appendix| command when writing an appendix at the end of your chapter or contribution. If there is only one the appendix is designated ``Appendix'', or ``Appendix 1'', or ``Appendix 2'', etc. if there is more than one.

\subsection{Calculating the Mean of a Weibull}\label{sec:pythagmeanweibull}

Letting $\mu_{\alpha,\beta,\gamma}$ denote the mean of $f(x;\alpha,\beta,\gamma)$,
we have \begin{eqnarray} \mu_{\alpha,\beta,\gamma} & \ = \ & \int_\beta^\infty x \cdot
\frac{\gamma}{\alpha} \left(\frac{x-\beta}{\alpha}\right)^{\gamma-1}
e^{-((x-\beta)/\alpha)^\gamma}{\rm d} x\nonumber\\ & = & \int_\beta^\infty  \alpha
\frac{x-\beta}{\alpha} \cdot \frac{\gamma}{\alpha}
\left(\frac{x-\beta}{\alpha}\right)^{\gamma-1}e^{-((x-\beta)/\alpha)^\gamma}{\rm d} x\ +\
\beta. \end{eqnarray} We change variables by setting $u =
\left(\frac{x-\beta}{\alpha}\right)^\gamma$. Then ${\rm d} u = \frac{\gamma}{\alpha}
\left(\frac{x-\beta}{\alpha}\right)^{\gamma-1}{\rm d} x$ and we have \begin{eqnarray}
\mu_{\alpha,\beta,\gamma} & \ = \ & \int_0^\infty \alpha u^{\gamma^{-1}} \cdot
e^{-u} {\rm d} u \ + \ \beta \nonumber\\ & = & \alpha \int_0^\infty e^{-u}
u^{1+\gamma^{-1}} \frac{{\rm d} u}{u} \ + \ \beta \nonumber\\ & = & \alpha
\Gamma(1+\gamma^{-1}) \ + \ \beta.\end{eqnarray}

%A similar calculation determines the variance.

%\subsection{The Independence of Runs Scored and Allowed}\label{sec:pythagindeprsra}

\subsection{Independence test with structural zeros}\label{sec:pythagappendixindstructuralzeros}

We describe the iterative procedure needed to handle the structural zeros. A good reference is Bishop and Fienberg [2].

Let Bin($k$) be the $k$\textsuperscript{th} bin used in the chi-squared test for independence. For each team's incomplete contingency table, let $O_{r,c}$ be the observed number of games where the number of runs scored is in Bin($r$) and runs allowed is in Bin($c$). As games cannot end in a tie, we have $O_{r,r} = 0$ for all $r$.

We construct the expected contingency table with entries $E_{r,c}$ using an iterative process to find the maximum likelihood estimators for each entry. For $1 \leq r,c \leq 12$, let
\begin{equation}
 E^{(0)}_{r,c} \  =\  \left\{
     \begin{array}{lr}
       1 & \ {\rm if}\ r \neq c\ \\
       0 & \ {\rm if}\ r = c,
     \end{array}
   \right.
\end{equation}
and let
\begin{equation}
 X_{r,+}\ =\ \sum_{c} O_{r,c}, \ \ \ X_{c,+}\ =\ \sum_{r} O_{r,c}.
\end{equation}
We then have that
\begin{equation}
E^{(\ell)}_{r,c} \  =\  \left\{
     \begin{array}{lr}
     E^{(\ell -1)}_{r,c}X_{r,+} / \sum_{c} E^{(\ell -1)}_{r,c}\ \ \ {\rm if }\ \ell\ {\rm is\ odd}\\
        E^{(\ell -1)}_{r,c}X_{c,+} / \sum_{r} E^{(\ell -1)}_{r,c}\  \ \ {\rm if }\ \ell\ {\rm is\ even.}
   \end{array}
   \right.
\end{equation}

The values of $E_{r,c}$ can be found by taking the limit as $\ell \to \infty$ of $E^{(\ell)}_{r,c}$, and typically the convergence is rapid. The statistic
\begin{equation}
\sum_{r, c \atop r \neq c} \frac{(E_{r,c} - O_{r,c})^2}{E_{r,c}}
\end{equation}
follows a chi-square distribution with $(11-1)^2 - 11 = 89$ degrees of freedom.

\subsection{Linearizing Pythagoras}\label{sec:pythaglinearizingpythag}

Unlike the argument in \S\ref{sec:pythaglinearization}, we do not assume knowledge of multivariable calculus and derive the linearization using just single variable methods. The calculations below are of interest in their own right, as they highlight good approximation techniques.

%These calculations are also done in an expanded, unpublished version of \emph{First Order Approximations of the Pythagorean Won-Loss Formula for Predicting MLB Teams Winning Percentages} (by Steven J. Miller and Kevin Dayaratna, By The Numbers -- The Newsletter of the SABR Statistical Analysis Committee \textbf{22} (2012), no 1, 15--19); see \url{http://web.williams.edu/Mathematics/sjmiller/public_html/math/papers/DM_LinPythag_App.pdf}.

We assume there is some exponent $\gamma$ such that the winning percentage, $\lwp$, is  \begin{equation} \lwp \ = \ \frac{\rs^\gamma}{\rs^\gamma+\ra^\gamma}, \end{equation} with $\rs$ and $\ra$ the total runs scored and allowed. We multiply the right hand side by $(1/\rs^\gamma)/(1/\rs^\gamma)$ and write $\ra^\gamma$ as $\rs^\gamma - (\rs^\gamma - \ra^\gamma)$, and find \begin{eqnarray} \lwp & \ = \ & \frac{1}{1 + \frac{\ra^\gamma}{\rs^\gamma}} \ = \ \left(1 + \frac{\ra^\gamma}{\rs^\gamma}\right)^{-1} \ = \ \left(1 + \frac{\rs^\gamma - (\rs^\gamma-\ra^\gamma)}{\rs^\gamma}\right)^{-1}  \nonumber\\ &=& \left(1 + 1 - \frac{\rs^\gamma-\ra^\gamma}{\rs^\gamma}\right)^{-1} \nonumber\\ &=& \left(2 \cdot \left(1 - \frac{\rs^\gamma-\ra^\gamma}{2\rs^\gamma}\right)\right)^{-1} \nonumber\\ &=& \frac12 \left(1 - \frac{\rs^\gamma-\ra^\gamma}{2\rs^\gamma}\right)^{-1}; \end{eqnarray} notice we manipulated the algebra to pull out a 1/2, which indicates an average team; thus the remaining factor is the fluctuations about average.

We now use the geometric series formula, which says that if $|r| < 1$ then \begin{equation} \frac1{1+r} \ = \ 1 + r + r^2 + r^3 + \cdots.\end{equation} We let $r = (\rs^\gamma-\ra^\gamma)/2\rs^\gamma$; since runs scored and runs allowed should be close to each other, the difference of their $\gamma$ powers divided by twice the number of runs scored should be small. Thus $r$ in our geometric expansion should be close to zero, and we find \begin{eqnarray}\label{eq:lwpfirstapprox} \lwp & \ = \ & \frac12\left(1 + \frac{\rs^\gamma-\ra^\gamma}{2\rs^\gamma} + \left(\frac{\rs^\gamma-\ra^\gamma}{2\rs^\gamma}\right)^2 + \left(\frac{\rs^\gamma-\ra^\gamma}{2\rs^\gamma}\right)^3 + \cdots\right) \nonumber\\ & \approx & .500 + \frac{\rs^\gamma-\ra^\gamma}{4\rs^\gamma}. \end{eqnarray}

We now make some approximations. We expect $\rs^\gamma-\ra^\gamma$ to be small, and thus $\frac{\rs^\gamma-\ra^\gamma}{2\rs}$ should be small. This means we only need to keep the constant and linear terms in the expansion. Note that if we only kept the constant term, there would be no dependence on points scored or allowed!

We need to do a little more analysis to obtain a formula that is linear in $\rs - \ra$. Let $\rt$ denote the average number of runs scored per team in the league. We can write $\rs = \rl + x_s$ and $\ra = \rl + x_a$, where it is reasonable to assume $x_s$ and $x_a$ are small relative to $\rt$. The Mean Value Theorem from Calculus says that if $f(x) = (\rt+x)^\gamma$, then \begin{equation} f(x_s) - f(x_a) \ = \ f'(x_c) (x_s - x_a), \end{equation} where $x_c$ is some intermediate point between $x_s$ and $x_a$. As $f'(x) = \gamma (\rt+x)^{\gamma-1}$, we find \begin{eqnarray} \rs^\gamma - \ra^\gamma & \ = \ & f(x_s) - f(x_a) \ = \  f'(x_c) (x_s - x_a) \ = \ \gamma (\rt+x_c)^{\gamma-1}(\rs - \ra), \nonumber\\ \end{eqnarray} as $x_s - x_a = \rs - \ra$. Substituting this into (\ref{eq:lwpfirstapprox}) gives \begin{eqnarray} \lwp & \ \approx \ & .500 + \frac{\gamma (\rt+x_c)^{\gamma-1} (\rs - \ra)}{4\rs^\gamma} \ = \  .500 + \frac{\gamma (\rt+x_c)^{\gamma-1}}{4\rs^\gamma} (\rs-\ra).\nonumber\\ \end{eqnarray}

We make one final approximation. We replace the factors of $\rt+x_c$ in the numerator and $\rs^\gamma$ in the denominator with $\rt^\gamma$, the league average, and reach \begin{equation} \lwp \ \approx \ .500 + \frac{\gamma}{4\rt} (\rs - \ra). \end{equation} Thus the simple linear approximation model reproduces the result from multivariable Taylor series, namely that the interesting coefficient ${\rm B}$ should be approximately $\gamma/(4\rt)$.

%%%%%%%%%%%%%%%%%%%%%%%%%%%%%%%%%%%%%%%%%%%%%%%%%%%%%%%%%%%%%%%%%%%%%%%%%%%%%%%%%%%%%%%%%%%%%%%%%%%%%%%%%%%%%%%%%%%%%%%%%%%%%%%%%%%%%%%%%%%%%%%
%%%%%%%%%%%%%%%%%%%%%%%%%%%%%%%%%%%%%%%%%%%%%%%%%%%%%%%%%%%%%%%%%%%%%%%%%%%%%%%%%%%%%%%%%%%%%%%%%%%%%%%%%%%%%%%%%%%%%%%%%%%%%%%%%%%%%%%%%%%%%%%
%%%%%%%%%%%%%%%%%%%%%%%%%%%%%%%%%%%%%%%%%%%%%%%%%%%%%%%%%%%%%%%%%%%%%%%%%%%%%%%%%%%%%%%%%%%%%%%%%%%%%%%%%%%%%%%%%%%%%%%%%%%%%%%%%%%%%%%%%%%%%%%

\section{References}

\begin{enumerate}

\item P. Birnbaum,  \emph{Sabermetric Research: Saturday, April 24, 2010}, see \hfill \\ \textcolor{blue}{\url{http://blog.philbirnbaum.com/2010/04/marginal-value-of-win-in-baseball.html}}.

\item Y. M. M. Bishop and S. E. Fienberg, \emph{Incomplete Two-Dimensional Contingency Tables}, Biometrics \textbf{25} (1969), no. 1, 119--128.

\item K. Dayaratna and S. J. Miller, \emph{First Order Approximations of the Pythagorean Won-Loss Formula for Predicting MLB Teams Winning Percentages}, By The Numbers -- The Newsletter of the SABR Statistical Analysis Committee \textbf{22} (2012), no 1, 15--19.

\item C. N. B. Hammond, W. P. Johnson and S. J. Miller, \emph{The James Function}, 2013, preprint.

\item H. Hundel, \emph{Derivation of James' Pythagorean Formula}, 2003; see \hfill \\ \textcolor{blue}{\url{https://groups.google.com/forum/\#!topic/rec.puzzles/O-DmrUljHds}}.

\item B. James, \emph{1981 Baseball Abstract}, self-published, Lawrence, KS, 1981.

\item M. Jones and L. Tappin, \emph{The Pythagorean Theorem of Baseball and Alternative Models}, The UMAP Journal 26.2, 2005.

\item V. Miller and S. J. Miller, \emph{Relieving and Readjusting Pythagoras}, preprint 2014.

\item S. J. Miller, \emph{A derivation of the Pythagorean Won-Loss Formula in baseball}, Chance Magazine \textbf{20} (2007), no. 1, 40--48 (an abridged version appeared in The Newsletter of the SABR Statistical Analysis Committee \textbf{16} (February 2006), no. 1, 17--22, and an expanded version is online at \textcolor{blue}{\url{http://arxiv.org/pdf/math/0509698}}).

\item N. Silver, \emph{Is Alex Rodriguez Overpaid}, in Baseball Between the Numbers: Why Everything You Know About the Game Is Wrong, by The Baseball Prospectus Team of Experts, Basic Books, 2006.
    
\item Wikipedia, \emph{Pythagorean Expectation}, \hfill \\ \textcolor{blue}{\url{http://en.wikipedia.org/wiki/Pythagorean\underline{\ \ }expectation}}.

\item Wikipedia, \emph{Weibull}, \textcolor{blue}{\url{http://en.wikipedia.org/wiki/Weibull\underline{\ \ }distribution}}.

\end{enumerate}

\end{document}